\documentclass[12pt]{amsart}
\usepackage{amsmath,amssymb,amsthm,bm,bbm}
\usepackage{graphicx,enumerate}
\usepackage{layout}
\usepackage{multicol}

\theoremstyle{plain}
\newtheorem{theorem}{Theorem}[section]

\newtheorem{proposition}[theorem]{Proposition}

\theoremstyle{definition}

\newtheorem{aremark}[theorem]{Added remark}

\newtheorem*{acknowledgments}{Acknowledgments}

\newcommand{\bZ}{\mathbbm{Z}}\newcommand{\bQ}{\mathbbm{Q}}
\newcommand{\bC}{\mathbbm{C}}

\title[On Noether's problem for cyclic groups of prime order]
{On Noether's problem for cyclic groups of prime order}
\author{Akinari Hoshi}
\thanks{
This work was partially supported by JSPS KAKENHI, 
the Grant-in-Aid for Scientific Research (C), 
Grant Number 25400027.
}
\begin{document}
\maketitle

\begin{abstract}
Let $k$ be a field and $G$ be a finite group 
acting on the rational function field $k(x_g\,|\,g\in G)$ by 
$k$-automorphisms $h(x_g)=x_{hg}$ for any $g,h\in G$. 
Noether's problem asks whether the invariant field $k(G)=k(x_g\,|\,g\in G)^G$ 
is rational (i.e. purely transcendental) over $k$. 
In 1974, Lenstra gave a necessary and sufficient condition 
to this problem for abelian groups $G$. 
However, even for the cyclic group $C_p$ of prime order $p$, 
it is unknown whether there exist infinitely many primes $p$ 
such that $\bQ(C_p)$ is rational over $\bQ$. 
Only known $17$ primes $p$ for which $\bQ(C_p)$ is rational over $\bQ$ 
are $p\leq 43$ and $p=61,67,71$. 
We show that for primes $p< 20000$, $\bQ(C_p)$ is not (stably) 
rational over $\bQ$ 
except for affirmative $17$ primes and undetermined $46$ primes. 
Under the GRH, the generalized Riemann hypothesis, we also confirm that 
$\bQ(C_p)$ is not (stably) rational over $\bQ$ for undetermined 
$28$ primes $p$ out of $46$. 
\end{abstract}

\section{Introduction}
Let $k$ be a field and $K$ be an extension field of $k$. 
A field $K$ is said to be {\it rational} over $k$ if $K$ is purely 
transcendental over $k$. 
A field $K$ is said to be {\it stably rational} over $k$ 
if the field $K(t_1,\ldots,t_n)$ is rational over $k$ for 
some algebraically independent elements $t_1,\ldots,t_n$ over $K$. 

Let $G$ be a finite group acting on the rational function field 
$k(x_g\,|\,g\in G)$ by $k$-automorphisms $h(x_g)=x_{hg}$ for any $g,h\in G$. 
We denote the fixed field $k(x_g\,|\,g\in G)^G$ by $k(G)$. 
Emmy Noether \cite{Noe13, Noe17} asked 
whether $k(G)$ is rational (= purely transcendental) over $k$. 
This is called Noether's problem for $G$ over $k$, 
and is related to the inverse Galois problem 
(see a survey paper of Swan \cite{Swa83} for details). 
Let $C_n$ be the cyclic group of order $n$. 

We define the following sets of primes: 
\begin{align*}
R=\{
&2, 3, 5, 7, 11, 13, 17, 19, 23, 29, 31, 37, 41, 43, 61, 67, 71\} 
{\rm\ (rational\ cases)},\\
U=\{
&251, 347, 587, 2459, 2819, 3299, 4547, 4787, 6659, 10667,\\ 
&12227, 14281, 15299, 17027, 17681, 18059, 18481, 18947
\} {\rm\ (undetermined\ cases)},\\
X=\{
&59, 83, 107, 163, %227, 
487, 677, 727, 1187, 1459, 2663, 3779, 4259,\\ %5507, %5939, 6563, 6899, 
%7187,\\
&7523, 8837, 10883, 11699, 12659, 12899, 13043, 13183, 13523,\\ 
&14243, 14387, 14723, 14867, 16547, 17939, 19379
\} {\rm\ (not\ rational\ cases\ under\ the\ GRH)}
\end{align*}
with $\#R=17$, $\#U=18$, $\#X=28$. 

The aim of this paper is to show the following theorem. 
\begin{theorem}\label{thmain}
Let $p< 20000$ be a prime. %number. 
If %either 
{\rm (i)} $p\not\in R\cup U\cup X$ or 
{\rm (ii)} under the GRH, the generalized Riemann hypothesis, 
$p\not\in R\cup U$, 
then $\bQ(C_p)$ is not stably rational over $\bQ$. 
\end{theorem}

%%%%%%%%%%%%%%%%%%%%%%%%%%%%%%%%%%%%%%%%%%%%%%%%%%%%%%%%%%%%%%%%%%%%%%%%%%%%%%%%%%%%%%%%%%
\begin{acknowledgments}
The author thanks Shizuo Endo and 
Ming-chang Kang for valuable discussions. 
He also thanks Keiichi Komatsu, Takashi Fukuda 
and John C. Miller for helpful comments. 
\end{acknowledgments}

%%%%%%%%%%%%%%%%
\section{Noether's problem for abelian groups}

We give a brief survey of Noether's problem for abelian groups. 
The reader is referred to Swan's survey papers \cite{Swa81} and \cite{Swa83}. 

\begin{theorem}[{Fischer \cite{Fis15}, see also Swan \cite[Theorem 6.1]{Swa83}}]\label{thFis}
Let $G$ be a finite abelian group with exponent $e$. 
Assume that {\rm (i)} either {\rm char} $k=0$ or {\rm char} $k>0$ with 
{\rm char} $k$ $\not{|}$ $e$, and 
{\rm (ii)} $k$ contains a primitive $e$-th root of unity. 
Then $k(G)$ is rational over $k$. 
\end{theorem}
\begin{theorem}[{Kuniyoshi \cite{Kun54, Kun55, Kun56}}]
Let $G$ be a $p$-group and $k$ be a field with {\rm char} $k=p>0$. 
Then $k(G)$ is rational over $k$. 
\end{theorem}
Masuda \cite{Mas55, Mas68} gave an idea to use a technique of Galois descent 
to Noether's problem for cyclic groups $C_p$ of order $p$. 
Let $\zeta_p$ be a primitive $p$-th root of unity, 
$L=\bQ(\zeta_p)$ and $\pi={\rm Gal}(L/\bQ)$. 
Then, by Theorem \ref{thFis}, we have $\bQ(C_p)=\bQ(x_1,\ldots,x_p)^{C_p}
=(L(x_1,\ldots,x_p)^{C_p})^{\pi}
%=(L(x_1,\ldots,x_p)^{\pi})^{C_p}
=L(y_0,\ldots,y_{p-1})^{\pi}=L(M)^{\pi}(y_0)$ 
where %$y_0=x_1+\cdots+x_p$ 
$y_0=\sum_{i=1}^p x_i$ is $\pi$-invariant, 
$M$ is free $\bZ[\pi]$-module and $\pi$ acts on $y_1,\ldots,y_{p-1}$ by 
$\sigma(y_i)=\prod_{j=1}^{p-1}y_j^{a_{ij}}$, 
$[a_{ij}]\in GL_n(\bZ)$ for any $\sigma\in\pi$. 
Thus the field $L(M)^\pi$ may be regarded as the function field 
of some algebraic torus of dimension $p-1$ (see e.g. \cite[Chapter 3]{Vos98}). 

\vspace*{3mm}
\begin{theorem}[{Masuda \cite{Mas55, Mas68}, see also \cite[Lemma 7.1]{Swa83}}]\label{thMasuda}~\\
{\rm (i)} $M$ is projective $\bZ[\pi]$-module of rank one;\\
{\rm (ii)} 
If $M$ is a permutation $\bZ[\pi]$-module, i.e. 
$M$ has a $\bZ$-basis which is permuted by $\pi$, 
then $L(M)^{\pi}$ is 
rational over $\bQ$. 
In particular, $\bQ(C_p)$ is rational over $\bQ$ 
for $p\leq 11$.\footnote{The author \cite[Chapter 5]{Hos05} 
generalized Theorem \ref{thMasuda} (ii) to Frobenius groups $F_{pl}$ of order $pl$ 
with $l\mid p-1$ ($p\leq 11$).}
\end{theorem}
Swan \cite{Swa69} gave the first negative solution to Noether's problem 
by investigating a partial converse to Masuda's result. 
\vspace*{3mm}
\begin{theorem}[{Swan \cite[Theorem 1]{Swa69}, Voskresenskii \cite[Theorem 2]{Vos70}}]\label{thSwan}
~\\
{\rm (i)} If $\bQ(C_p)$ is rational over $\bQ$, 
then there exists $\alpha\in \bZ[\zeta_{p-1}]$ such that 
$N_{\bQ(\zeta_{p-1})/\bQ}(\alpha)=\pm p$;\\
{\rm (ii)} {\rm (Swan)} $\bQ(C_{47})$, $\bQ(C_{113})$ and 
$\bQ(C_{233})$ are not rational over $\bQ$;\\
{\rm (iii)} {\rm (Voskresenskii)} $\bQ(C_{47})$, $\bQ(C_{167})$, 
$\bQ(C_{359})$, $\bQ(C_{383})$, $\bQ(C_{479})$, $\bQ(C_{503})$ and 
$\bQ(C_{719})$ are not rational over $\bQ$.
\end{theorem}
\begin{theorem}[{Voskresenskii \cite[Theorem 1]{Vos71}}]\label{thVos}
$\bQ(C_p)$ is rational over $\bQ$ if and only if 
there exists $\alpha\in \bZ[\zeta_{p-1}]$ such that 
$N_{\bQ(\zeta_{p-1})/\bQ}(\alpha)=\pm p$.
\end{theorem}

Hence if the cyclotomic field $\bQ(\zeta_{p-1})$ has class number one, 
then $\bQ(C_p)$ is rational over $\bQ$. 
However, it is known that such primes are exactly $p\leq 43$ and $p=61,67,71$ 
(see Masley and Montgomery \cite[Main theorem]{MM76} or 
Washington's book \cite[Chapter 11]{Was97}). 

Endo and Miyata \cite{EM73} refined Masuda-Swan's method 
and gave some further consequences on Noether's problem when $G$ is abelian 
(see also \cite{Vos73}). 
\begin{theorem}[Endo and Miyata {\cite[Theorem 2.3]{EM73}}]\label{thEM23}
Let $G_1$ and $G_2$ be finite groups and $k$ be a field with {\rm char} $k=0$. 
If $k(G_1)$ and $k(G_2)$ are rational $($resp. stably rational$)$ 
over $k$, then 
$k(G_1\times G_2)$ is rational $($resp. stably rational$)$ over 
$k$.\footnote{Kang and Plans \cite[Theorem 1.3]{KP09} showed that 
Theorem \ref{thEM23} is also valid for any field $k$.}
\end{theorem}

The converse of Theorem \ref{thEM23} does not hold for general 
$k$, see e.g. 
%\cite[Proposition 4.6]{EM73} and also 
Theorem \ref{thEM73m} below. 

\begin{theorem}[Endo and Miyata {\cite[Theorem 3.1]{EM73}}]
Let $p$ be an odd prime and $l$ be a positive integer. 
Let $k$ be a field with {\rm char} $k=0$ and 
$[k(\zeta_{p^l}):k]=p^{m_0}d_0$ with $0\leq m_0\leq l-1$ and 
$d_0\mid p-1$. 
Then the following conditions are equivalent:\\
{\rm (i)} For any faithful $k[C_{p^l}]$-module $V$, 
$k(V)^{C_{p^l}}$ is rational over $k$;\\
{\rm (ii)} $k(C_{p^l})$ is rational over $k$;\\
%{\rm (iii)} $I_k(p^l)$ is a quasi-permutation $\Pi_k(p^l)$-module;\\
%{\rm (iv)} $J_k(p^l)$ is a principal ideal of $\bZ[\zeta_{p^{m_0}d_0}]$;\\
{\rm (iii)} There exists $\alpha\in \bZ[\zeta_{p^{m_0}d_0}]$ 
such that 
\[
N_{\bQ(\zeta_{p^{m_0}d_0})/\bQ}(\alpha)=
\begin{cases}
\pm p& m_0>0\\
\pm p^l & m_0=0.
\end{cases}
\]
Further suppose that $m_0>0$. 
Then the above conditions are equivalent to each of 
the following conditions:\\
{\rm (i${}^\prime$)} For any $k[C_{p^l}]$-module $V$, 
$k(V)^{C_{p^l}}$ is rational over $k$;\\
{\rm (ii${}^\prime$)} For any $1\leq l^\prime\leq l$, 
$k(C_{p^{l^\prime}})$ is rational over $k$.
\end{theorem}
\begin{theorem}[Endo and Miyata {\cite[Proposition 3.2]{EM73}}]
Let $p$ be an odd prime and $k$ be a field with {\rm char} $k=0$. 
If $k$ contains $\zeta_p+\zeta_p^{-1}$, then $k(C_{p^l})$ is 
rational over $k$ for any $l$. 
In particular, $\bQ(C_{3^l})$ is rational over $\bQ$ for any $l$. 
\end{theorem}
\vspace*{0mm}
\begin{theorem}[Endo and Miyata {\cite[Proposition 3.4, Corollary 3.10]{EM73}}]
~\\
{\rm (i)} For primes $p\leq 43$ and $p=61,67,71$, 
$\bQ(C_p)$ is rational over $\bQ$;\\
{\rm (ii)} For $p=5,7$, $\bQ(C_{p^2})$ is rational over $\bQ$;\\
{\rm (iii)} For $l\geq 3$, $\bQ(C_{2^l})$ is not stably rational over $\bQ$. 
\end{theorem}
\begin{theorem}[Endo and Miyata {\cite[Theorem 4.4]{EM73}}]\label{thEM73m}
Let $G$ be a finite abelian group of odd order and $k$ be a 
field with {\rm char} $k=0$. Then there exists an integer $m>0$ 
such that $k(G^m)$ is rational over $k$. 
\end{theorem}
\begin{theorem}[Endo and Miyata {\cite[Theorem 4.6]{EM73}}]
Let $G$ be a finite abelian group. 
Then $\bQ(G)$ is rational over $\bQ$ if and only if 
$\bQ(G)$ is stably rational over $\bQ$.
%The following are equivalent:\\
%{\rm (i)} $\bQ(G)$ is rational over $\bQ$;\\
%{\rm (ii)} $\bQ(G)$ is stably rational over $\bQ$.
\end{theorem}

Ultimately, 
Lenstra \cite{Len74} gave a necessary and sufficient condition 
of Noether's problem for abelian groups. 
\begin{theorem}[Lenstra {\cite[Main Theorem, Remark 5.7]{Len74}}]
Let $k$ be a field and $G$ be a finite abelian group. 
Let $k_{\rm cyc}$ be the maximal cyclotomic extension of $k$ in 
an algebraic closure. 
For $k\subset K\subset k_{\rm cyc}$, we assume that 
$\rho_K={\rm Gal}(K/k)=\langle\tau_k\rangle$ is finite cyclic.
Let $p$ be an odd prime with $p\neq {\rm char}$ $k$ and $s\geq 1$ be an integer. 
Let $\mathfrak{a}_K(p^s)$ be a $\bZ[\rho_K]$-ideal defined by 
\begin{align*}
\mathfrak{a}_K(p^s)=
\begin{cases}
\bZ[\rho_K] & {\rm if}\ K\neq k(\zeta_{p^s})\\
(\tau_K-t,p) & {\rm if}\ K=k(\zeta_{p^s})\ {\rm where}\ 
t\in\bZ\ {\rm satisfies}\ \tau_K(\zeta_p)=\zeta_p^t
\end{cases}
\end{align*}
and put $\mathfrak{a}_K(G)=\prod_{p,s}\mathfrak{a}_K(p^s)^{m(G,p,s)}$ 
where $m(G,P,s)=\dim_{\bZ/p\bZ}(p^{s-1}G/p^sG)$. 
Then the following conditions are equivalent:\\
{\rm (i)} $k(G)$ is rational over $k$;\\
{\rm (ii)} $k(G)$ is stably rational over $k$;\\
{\rm (iii)} for $k\subset K\subset k_{\rm cyc}$, 
the $\bZ[\rho_K]$-ideal $\mathfrak{a}_K(G)$ is principal 
and if {\rm char} $k\neq 2$, then $k(\zeta_{r(G)})/k$ is cyclic extension 
where $r(G)$ is the highest power of $2$ dividing the exponent of $G$. 
\end{theorem}
\begin{theorem}[Lenstra {\cite[Corollary 7.2]{Len74}, see also \cite[Proposition 2, Corollary 3]{Len80}}]
Let $n$ be a positive integer. 
Then the following conditions are equivalent:\\
{\rm (i)} $\bQ(C_n)$ is rational over $\bQ$;\\
{\rm (ii)} $k(C_n)$ is rational over $k$ for any field $k$;\\
{\rm (iii)} $\bQ(C_{p^s})$ is rational over $\bQ$ 
for any $p^s$ $||$ $n$;\\
{\rm (iv)} $8$ $\not{|}$ $n$ and for any $p^s$ $||$ $n$, 
%there exists $\alpha\in \bZ[\zeta_{(p-1)p^{s-1}}]$ such that 
%$N_{\bQ(\zeta_{(p-1)p^{s-1}})/\bQ}(\alpha)=\pm p$. 
there exists $\alpha\in \bZ[\zeta_{\varphi(p^s)}]$ such that 
$N_{\bQ(\zeta_{\varphi(p^s)})/\bQ}(\alpha)=\pm p$. 
\end{theorem}
\begin{theorem}[{Lenstra \cite[Corollary 7.6]{Len74}, see also 
\cite[Proposition 6]{Len80}}]
Let $k$ be a field which is finitely generated over its prime field. 
Let $P_k$ be the set of primes  $p$ for which 
$k(C_p)$ is rational over $k$. Then $P_k$ has Dirichlet density $0$ 
inside the set of all primes $p$. 
In particular, 
\[
\lim_{x\rightarrow\infty}\frac{\pi^*(x)}{\pi(x)}=0
\]
where $\pi(x)$ is the number of primes $p\leq x$, 
and $\pi^*(x)$ is the number of primes $p\leq x$ 
for which $\bQ(C_p)$ is rational over $\bQ$. 
\end{theorem}
\begin{theorem}[Lenstra {\cite[Proposition 4]{Len80}}]
Let $p$ be a prime and $s\geq 2$ be an integer. 
Then $\bQ(C_{p^s})$ is rational over $\bQ$ if and only if 
$p^s\in\{2^2, 3^m, 5^2,7^2\mid m\geq 2\}$. 
\end{theorem}
However, even in the case $k=\bQ$ and $p< 1000$, 
there exist primes $p$ (e.g. $59$, $83$, $107$, $251$, etc.) 
such that the rationality of 
$\bQ(C_p)$ over $\bQ$ is undetermined (see Theorem \ref{thmain}). 
Moreover, 
we do not know whether there exist infinitely many primes $p$ 
such that $\bQ(C_p)$ is rational over $\bQ$. 
This derives a motivation of this paper. %(Theorem \ref{thmain}). 

We finally remark that 
although $\bC(G)$ is rational over $\bC$ for any abelian group $G$ by Theorem \ref{thFis}, 
Saltman \cite{Sal84} gave a $p$-group $G$ of order $p^9$ for which 
Noether's problem has a negative answer over $\bC$
%a negative answer to Noether's problem 
%for non-abelian $p$-group $G$ over algebraically closed field $k$ 
using the unramified Brauer group $B_0(G)$. 
Indeed, one can see that $B_0(G)\neq 0$ implies that $\bC(G)$ is not 
retract rational over $\bC$, and hence not (stably) rational over $\bC$.

\begin{theorem}\label{thSB}
Let $p$ be any prime.  \\
{\rm (i)}\ $(${\rm Saltman} \cite{Sal84}$)$ 
There exists a meta-abelian $p$-group $G$ of order $p^9$ such that $B_0(G)\neq 0$;\\
{\rm (ii)}\ $(${\rm Bogomolov} \cite{Bog88}$)$ 
There exists a group $G$ of order $p^6$ such that $B_0(G)\neq 0$;\\
%Thus $k(G)$ is not rational over $k$. \\
{\rm (iii)}\ $(${\rm Moravec} \cite{Mor12}$)$ 
There exist exactly $3$ groups $G$ of order $3^5$ such that $B_0(G)\neq 0$;\\
{\rm (iv)}\ $(${\rm Hoshi, Kang and Kunyavskii} \cite{HKK13}$)$ 
For groups $G$ of order $p^5$ $(p\ge 5)$, 
$B_0(G)\ne 0$ if and only if $G$ belongs to the isoclinism family
$\Phi_{10}$. 
There exist exactly $1+\gcd\{4,p-1\}+\gcd \{3,p-1\}$ 
groups $G$ of order $p^5$ $(p\geq 5)$ such that  $B_0(G)\neq 0$.

In particular, for the cases where $B_0(G)\neq 0$, 
%if $k$ is an algebraically closed field with {\rm char} $k\neq p$, then 
$\bC(G)$ is not retract rational over $\bC$. 
Thus $\bC(G)$ is not $($stably$)$ rational over $\bC$. 
\end{theorem}

The reader is referred to \cite{CHKK10, Kan12, HKK13, BB13, Kan13, Kan14} 
and the references therein 
for more recent progress about unramified Brauer groups 
and retract rationality of fields. 

%81, 83
%%%%%%%%%%%%%%%%%%%%%%%%%%%%%%%%%%%%%%%%%%%%%%%%%
\section{Proof of Theorem \ref{thmain}}

By Swan's theorem (Theorem \ref{thSwan}), 
Noether's problem for $C_p$ over $\bQ$ has a negative answer 
if the norm equation $N_{F/\bQ}(\alpha)=\pm p$ 
has no integral solution for some intermediate field 
$\bQ\subset F\subset\bQ(\zeta_{p-1})$ with $[F:\bQ]=d$. 
When $d=2$, Endo and Miyata gave the following result: 

\begin{proposition}[Endo and Miyata {\cite[Proposition 3.6]{EM73}}]\label{propEM}
Let $p$ be an odd prime satisfying one of the following conditions:\\
{\rm (i)} $p=2q+1$ where $q\equiv -1\pmod{4}$, $q$ is square-free, 
and any of $4p-q$ and $q+1$ is not square;\\
{\rm (ii)} $p=8q+1$ where $q\not\equiv -1\pmod{4}$, $q$ is square-free, 
and any of $p-q$ and $p-4q$ is not square. 
Then $\bQ(C_p)$ is not rational over $\bQ$. 
\end{proposition}

By Proposition \ref{propEM} and case-by-case analysis for $d=2$ and $d=4$, 
Endo and Miyata confirmed that Noether's problem 
for $C_p$ over $\bQ$ has a negative answer for some primes $p< 2000$ 
(\cite[Appendix]{EM73}). 
We will give all primes $p< 20000$ which satisfy 
Proposition \ref{propEM} (i) (resp. (ii))  
as in Table $1$ (resp. Table $2$) in Appendix 
(Section \ref{seApp}). 
Tables $1$ and $2$ agree with our computational result 
(cf. Section \ref{seResult}).

In general, we may have to check all intermediate fields 
$\bQ\subset F\subset\bQ(\zeta_{p-1})$ with degree $2\leq d\leq \varphi(p-1)$. 
However, fortunately, it turns out that for many cases, 
we can determine the rationality of $\bQ(C_p)$ 
by some intermediate field $F$ of low degree $d\leq 8$ (see Section \ref{seResult}). 

We make an algorithm using the computer software PARI/GP \cite{PARI2} 
for general $d\mid p-1$. 
We can prove Theorem \ref{thmain} 
by the following function {\tt NP(j,\{GRH\},\{L\})} of PARI/GP 
which may determine whether Noether's problem for $C_{p_j}$ over $\bQ$ 
has a positive answer for the $j$-th prime $p_j$ 
unconditionally, i.e. without the GRH, 
if ${\tt GRH}=0$ 
(resp. under the GRH if ${\tt GRH}=1$). 

\vspace*{10mm}
{\small 
\begin{verbatim}
NP(j,GRH=0,L=[1,1])=
{
 local(p,Z,G,C,d1,d2,B,S,k);
 p=prime(j);
 Z=znstar(p-1);
 G=matdiagonal(Z[2]);
 C=[];d1=[0,0];d2=[0,0];k=0;
 forsubgroup(H=G,p-1,C=concat(C,galoissubcyclo(Z,mathnf(concat(G,H)))));
 C=Set(C);
 if(GRH==0,
  for(i=L[1],#C,B=bnfinit(C[i]);S=bnfisintnorm(B,p);
   if(S==[],k=i;d1=[poldegree(C[i]),bnfcertify(B)];break)
  );
  for(i=L[2],#C,B=bnfinit(C[i]);S=bnfisintnorm(B,-p);
   if(S==[],if(i==k,d2=[poldegree(C[i]),1],
   d2=[poldegree(C[i]),bnfcertify(B)]);break)
  );
 );
 if(GRH==1,
  for(i=L[1],#C,B=bnfinit(C[i]);S=bnfisintnorm(B,p);
   if(S==[],d1=[poldegree(C[i]),1];break)
  );
  for(i=L[2],#C,B=bnfinit(C[i]);S=bnfisintnorm(B,-p);
   if(S==[],d2=[poldegree(C[i]),1];break)
  );
 );
if([d1[2],d2[2]]==[1,1],return([d1[1],d2[1],GRH]),return([Ratioal,GRH]))
}
\end{verbatim}
}
\vspace*{10mm}

{\tt NP(j,\{GRH\},\{L\})} returns the list $[d_+,d_-,{\tt GRH}]$ for the 
$j$-th prime $p_j$ and $L=\{l_+,l_-\}$ 
without the GRH if ${\tt GRH}=0$ 
(resp. under the GRH if ${\tt GRH}=1$) 
where 
$d_\pm=[K_{\pm,i}:\bQ]$ 
if the norm equation $N_{K_{\pm,i}/\bQ}(\alpha)=\pm p_j$ 
has no integral solution for some $i$-th subfield 
$\bQ\subset K_{\pm,i}\subset\bQ(\zeta_{p_j-1})$ with 
$i\geq l_\pm$, $d_\pm ={\tt Rational}$ if the norm equation 
$N_{\bQ(\zeta_{p_j-1})/\bQ}(\alpha)=\pm p_j$ 
has an integral solution. 
The second and third inputs ${\tt \{GRH\}}$, ${\tt \{L\}}$ 
%of the function {\tt NP} 
may be omitted. 
If they are omitted, the function {\tt NP} runs as 
${\tt GRH}=0$ and ${\tt L}=[1,1]$, 
namely it works without the GRH and for all subfields 
$\bQ\subset K_{\pm,i}\subset\bQ(\zeta_{p_j-1})$ respectively.
%\\
%\vspace*{3mm}
\newpage

We further define the set of primes:
%\begin{align*}
%S_0=\{
%&783,962,1177,1180,1279,1328,1380,1425,1428,1458,1467,1553,1572,1584,\\
%&1651,1661,1761,1831,1840,1884,1886,1948,1974,2020,2028,2030,2041,\\
%&2044,2072,2109,2136,2158,2171,2180,2205,2214,2221,2245,2258,2262\},\\
%S_1=\{&1404,1513,1535,1554,1673,1723,2057,2193\}.
%T=45,49,94,220,241,256,276,317,376,427,780,848,
%887,995,1101,1402,1616,1707,1800,2066,2074,2224
%\end{align*}
\begin{align*}
S_0=\{&5987, 7577, 9497, 9533, 10457, 10937, 11443, 11897, 11923, 12197,\\
&12269, 13037, 13219, 13337, 13997, 14083, 15077, 15683, 15773, 16217,\\
&16229, 16889, 17123, 17573, 17657, 17669, 17789, 17827, 18077, 18413,\\
&18713, 18979, 19139, 19219, 19447, 19507, 19577, 19843, 19973, 19997\},\\
S_1=\{&11699, 12659, 12899, 13043, 14243, 14723, 17939, 19379\}\subset X,\\
T_0=\{&197, 227, 491, 1373, 1523, 1619, 1783, 2099, 2579, 2963, 
%%%%
5507,\\ 
%%%%
&5939, 6563, 6899, 
%%%%
7187, 
%%%%
7877, %11681, 13681, 
14561, %15401, 
18041, 18097, 19603\},\\
T_1=\{&8837\}\subset X
\end{align*}
with $\# S_0=40$, $\#S_1=8$, $\# T_0=20$, $\# T_1=1$. 

We split the proof of Theorem \ref{thmain} ($p<20000)$ into three parts:\\
(i)  $p\in S_0\cup S_1$;\\
(ii) $p\in T_0\cup T_1$;\\
(iii) $p\not\in U\cup S_0\cup S_1\cup T_0\cup T_1$. 

We will treat the cases (i), (ii), (iii) 
in Subsections \ref{ss31}, \ref{ss32}, \ref{ss33} respectively.

%%%%%%%%%%%%%%%%%%%%%%%%%%%%%%%%%%%%%%%%%%%
\subsection{Case $p\in S_0\cup S_1$}\label{ss31}
~\\

When $p_j\in S_0\cup S_1$, 
we should take a suitable list ${\tt L}$ for 
the function {\tt NP(j,GRH,L)}. 
For $p_j\in S_0$ (resp. $p_j\in S_1$), 
we may take the following ${\tt L}$ in $L_0$ (resp. $L_1$) respectively: 
\vspace*{3mm}
%\newpage

{\small 
\begin{verbatim}
L0=[[20,19],[1,3],[1,3],[9,1],[1,3],[1,3],[1,3],[1,3],[1,3],[3,1],
    [1,3],[9,3],[1,3],[1,3],[1,3],[1,3],[10,1],[4,1],[8,3],[1,3],
    [3,1],[1,3],[1,3],[1,3],[1,3],[1,3],[9,3],[1,3],[9,3],[9,3],
    [1,3],[1,3],[1,3],[1,3],[1,3],[1,3],[1,3],[1,3],[3,1],[9,3]];
L1=[[3,1],[3,1],[1,3],[1,3],[1,3],[41,1],[4,1],[3,1]];
\end{verbatim}
}
\vspace*{3mm}

Let $S_{0,j}$ (resp. $S_{1,j}$) be the index set $\{j\}$ of the 
set $S_0=\{p_j\}$ (resp. $S_1$). 
\vspace*{3mm}

{\small
\begin{verbatim}
S0j=[783,962,1177,1180,1279,1328,1380,1425,1428,1458,
     1467,1553,1572,1584,1651,1661,1761,1831,1840,1884,
     1886,1948,1974,2020,2028,2030,2041,2044,2072,2109,
     2136,2158,2171,2180,2205,2214,2221,2245,2258,2262];
S1j=[1404,1513,1535,1554,1673,1723,2057,2193];
\end{verbatim}
}
\vspace*{3mm}

For example, we take $p_j=5987\in S_0$ with $j=783$. 
Then {\tt NP(783,0)} does not work well 
in a reasonable time. 
%{\small
%\begin{verbatim}
%gp > NP(783,0)
%\end{verbatim}
%}
However, {\tt NP(783,0,[20,19])} %works. 
returns an answer 
%We may get the following 
in a few seconds.

\vspace*{3mm}
{\small
\begin{verbatim}
gp > NP(783,0,[20,19])
[ 8, 8, 0]
\end{verbatim}
}
\vspace*{3mm}

Namely, 
the norm equation $N_{K_{+,i}/\bQ}(\alpha)=p_j$ 
has no integral solution for some $i$-th subfield 
$\bQ\subset K_{+,i}\subset\bQ(\zeta_{p_j-1})$ with 
$i\geq 20$ and $[K_{+,i}:\bQ]=8$, 
and $N_{K_{-,i}/\bQ}(\alpha)=-p_j$ 
has no integral solution for some $i$-th subfield 
$\bQ\subset K_{-,i}\subset\bQ(\zeta_{p_j-1})$ with 
$i\geq 19$ and $[K_{-,i}:\bQ]=8$.

We can confirm Theorem \ref{thmain} 
for $p_j\in S_0$ (resp. $p_j\in S_1)$ unconditionally, i.e. without the GRH, 
(resp. under the GRH) using {\tt NP(j,GRH,L)} with ${\tt GRH}=0$ 
(resp. ${\tt GRH}=1$) as follows. 
%\newpage
\vspace*{3mm}

{\small
\begin{verbatim}
gp > #
   timer = 1 (on)
gp > allocatemem(2^4*10^8)
  ***   Warning: new stack size = 1600000000 (1525.879 Mbytes).
gp > for(m=1,#L0,print([[S0j[m],prime(S0j[m])],NP(S0j[m],0,L0[m])]))
\end{verbatim}
}
\vspace*{-3mm}
%{\scriptsize
{\tiny
\begin{multicols}{3}
\begin{verbatim}
[[783, 5987], [8, 8, 0]]
[[962, 7577], [2, 2, 0]]
[[1177, 9497], [2, 2, 0]]
[[1180, 9533], [6, 2, 0]]
[[1279, 10457], [2, 2, 0]]
[[1328, 10937], [2, 2, 0]]
[[1380, 11443], [2, 2, 0]]
[[1425, 11897], [2, 2, 0]]
[[1428, 11923], [2, 2, 0]]
[[1458, 12197], [4, 2, 0]]
[[1467, 12269], [6, 2, 0]]
[[1553, 13037], [6, 2, 0]]
[[1572, 13219], [2, 2, 0]]
[[1584, 13337], [2, 2, 0]]
[[1651, 13997], [6, 2, 0]]
[[1661, 14083], [2, 2, 0]]
[[1761, 15077], [6, 2, 0]]
[[1831, 15683], [5, 5, 0]]
[[1840, 15773], [6, 2, 0]]
[[1884, 16217], [2, 2, 0]]
[[1886, 16229], [4, 2, 0]]
[[1948, 16889], [2, 2, 0]]
[[1974, 17123], [2, 2, 0]]
[[2020, 17573], [2, 2, 0]]
[[2028, 17657], [2, 2, 0]]
[[2030, 17669], [2, 2, 0]]
[[2041, 17789], [6, 2, 0]]
[[2044, 17827], [2, 2, 0]]
[[2072, 18077], [6, 2, 0]]
[[2109, 18413], [6, 2, 0]]
[[2136, 18713], [2, 2, 0]]
[[2158, 18979], [2, 2, 0]]
[[2171, 19139], [2, 2, 0]]
[[2180, 19219], [2, 2, 0]]
[[2205, 19447], [2, 2, 0]]
[[2214, 19507], [2, 2, 0]]
[[2221, 19577], [2, 2, 0]]
[[2245, 19843], [2, 2, 0]]
[[2258, 19973], [4, 2, 0]]
[[2262, 19997], [6, 2, 0]]
time = 2h, 6min, 4,969 ms.
\end{verbatim}
\end{multicols}\vspace*{-3mm}
}
\begin{verbatim}
gp > for(m=1,#L1,print([[S1j[m],prime(S1j[m])],NP(S1j[m],1,L1[m])]))
\end{verbatim}\vspace*{-3mm}
%{\scriptsize
{\tiny
\begin{multicols}{3}
\begin{verbatim}
[[1404, 11699], [8, 8, 1]]
[[1513, 12659], [8, 8, 1]]
[[1535, 12899], [16, 16, 1]]
[[1554, 13043], [8, 8, 1]]
[[1673, 14243], [16, 16, 1]]
[[1723, 14723], [16, 2, 1]]
[[2057, 17939], [8, 8, 1]]
[[2193, 19379], [8, 8, 1]]
time = 1h, 27min, 21,610 ms.
\end{verbatim}
\end{multicols}
}

The computations of this paper were done 
on a machine with Intel Xeon E5-2687W (3.10 GHz, 64 GB RAM, Windows).
\vspace*{3mm}

%%%%%%%%%%%%%%%%%%%%%%%%%%%%%%%%%%%%%%%%%%%
\subsection{Case $p\in T_0\cup T_1$}\label{ss32}
~\\

When $p_j\in T_0\cup T_1$, because the computation of 
{\tt NP(j,GRH)} may take more time and memory resources, 
we will do that by case-by-case analysis. 
We can confirm Theorem \ref{thmain} 
for $p_j\in T_0$ (resp. $p_j\in T_1)$ unconditionally 
(resp. under the GRH) using {\tt NP(j,GRH)} with ${\tt GRH}=0$ 
(resp. ${\tt GRH}=1$) as follows. 
In particular, 
for two primes $p_j=5507$ with $j=728$ and $p_j=7187$ with $j=918$, 
it takes about 55 days and 45 days respectively in our computation.

\vspace*{3mm}
{\tiny 
\begin{multicols}{3}
\begin{verbatim}
gp > #
   timer = 1 (on)
gp > allocatemem(2^4*10^8)
  ***   Warning: new stack size = 
1600000000 (1525.879 Mbytes).
gp > j=45;[[j,prime(j)],NP(j,0)]
time = 37min, 41,598 ms. 
[[45, 197], [14, 2, 0]]
gp > j=49;[[j,prime(j)],NP(j,0)]
time = 210h, 36min, 48,779 ms.
[[49, 227], [16, 16, 0]]
gp > j=94;[[j,prime(j)],NP(j,0)]
time = 2h, 33min, 16.867 ms.
[[94, 491], [14, 2, 0]]
gp > j=220;[[j,prime(j)],NP(j,0)]
time = 37min, 52,659 ms.
[[220, 1373], [14, 2, 0]]
gp > j=241;[[j,prime(j)],NP(j,0)]
time = 3h, 17min, 52,181 ms.
[[241, 1523], [8, 8, 0]]
gp > j=256;[[j,prime(j)],NP(j,0)]
time = 2h, 39min, 35,993 ms.
[[256, 1619], [8, 8, 0]]
gp > j=276;[[j,prime(j)],NP(j,0)]
time = 8h, 38min, 13,006 ms.
[[276, 1783], [18, 2, 0]]
gp > j=317;[[j,prime(j)],NP(j,0)]
time = 7h, 37min, 3,212 ms.
[[317, 2099], [8, 8, 0]]
gp > j=376;[[j,prime(j)],NP(j,0)]
time = 3h, 44min, 13,917 ms.
[[376, 2579], [8, 8, 0]]
gp > j=427;[[j,prime(j)],NP(j,0)]
time = 9h, 39min, 45,264 ms.
[[427, 2963], [8, 8, 0]]
gp > j=780;[[j,prime(j)],NP(j,0)]
 *** bnfcertify: Warning: Zimmert's 
bound is large (2500735916), 
certification will take a long time.
time = 68h, 28min, 31,106 ms.
[[780, 5939], [8, 8, 0]]
gp > j=848;[[j,prime(j)],NP(j,0)]
time = 183h, 47min, 42,355 ms. 
[[848, 6563], [12, 12, 0]]
gp > j=887;[[j,prime(j)],NP(j,0)]
 *** bnfcertify: Warning: Zimmert's 
bound is large (4225307497), 
certification will take a long time.
time = 160h, 57min, 42,981 ms.
[[887, 6899], [8, 8, 0]]
gp > j=995;[[j,prime(j)],NP(j,0)]
time = 57min, 20,134 ms.
[[995, 7877], [10, 2, 0]]
gp > j=1707;[[j,prime(j)],NP(j,0)]
time = 1h, 40min, 17,426 ms.
[[1707, 14561], [2, 2, 0]]
gp > j=2066;[[j,prime(j)],NP(j,0)]
time = 44min, 33,330 ms.
[[2066, 18041], [2, 2, 0]]
gp > j=2074;[[j,prime(j)],NP(j,0)]
time = 35min, 5,584 ms.
[[2074, 18097], [2, 2, 0]]
gp > j=2224;[[j,prime(j)],NP(j,0)]
time = 8h, 24min, 24,406 ms.
[[2224, 19603], [18, 2, 0]]

gp > j=728;[[j,prime(j)],NP(j,0)]
[[728, 5507], [8, 8, 0]]
gp > j=918;[[j,prime(j)],NP(j,0)]
 *** bnfcertify: Warning: Zimmert's 
bound is large (4875648631), 
certification will take a long time.
[[918, 7187], [8, 8, 0]]

gp > j=1101;[[j,prime(j)],NP(j,1)]
time = 4h, 16,841 ms.
[[1101, 8837], [46, 2, 1]]
\end{verbatim}
\end{multicols}
}
%\vspace*{3mm}

%%%%%%%%%%%%%%%%%%%%%%%%%%%%%%%%%%%%%%%%%%%%%%%%%%%%%%%%%%%%%%
\subsection{Case $p\not\in U\cup S_0\cup S_1\cup T_0\cup T_1$}\label{ss33}
~\\

When $p_j\not\in U\cup S_0\cup S_1\cup T_0\cup T_1$, 
we just apply the function {\tt NP(j,GRH)}. 

Let $U_j$ (resp. $X_j$, $T_{0,j}$, $T_{1,j}$) 
be the index set $\{j\}$ of $U=\{p_j\}$ 
(resp. $X$, $T_{0}$, $T_{1}$).

\vspace*{3mm}
{\small 
\begin{verbatim}
Uj=[54,69,107,364,410,463,616,643,858,1302,
    1461,1676,1787,1963,2031,2070,2117,2155];
Xj=[17,23,28,38,93,123,129,195,232,386,526,584,
    953,1101,1323,1404,1513,1535,1554,1569,1602,
    1673,1685,1723,1741,1915,2057,2193];
T0j=[45,49,94,220,241,256,276,317,376,427,728,
     780,848,887,918,995,1707,2066,2074,2224];
T1j=[1101];
\end{verbatim}
}
\vspace*{3mm}

Then we can confirm Theorem \ref{thmain} for 
$p_j\not\in U\cup S_0\cup S_1\cup T_0\cup T_1$ 
unconditionally (resp. under the GRH) when $p_j\not\in X$ (resp. $p_j\in X$) 
using {\tt NP(j,GRH)} with ${\tt GRH}=0$ (resp. ${\tt GRH}=1$). 
The following function {\tt NPs(s,e)} performs it for each $p_s\leq p_j\leq p_e$ 
and it returns {\tt "?"} (resp. {\tt "S0"}, {\tt "S1"}, 
{\tt "T0"}, {\tt "T1"}) when $p_j\in U$ (resp. $S_0$, $S_1$, $T_0$, $T_1$). 

\vspace*{3mm}
{\small 
\begin{verbatim}
NPs(s,e)=
{
for(j=s,e,
if(setsearch(Uj,j),print([[j,prime(j)],"?"]),
 if(setsearch(S0j,j),print([[j,prime(j)],"S0"]),
  if(setsearch(S1j,j),print([[j,prime(j)],"S1"]),
   if(setsearch(T0j,j),print([[j,prime(j)],"T0"]),
    if(setsearch(T1j,j),print([[j,prime(j)],"T1"]),
     if(setsearch(Xj,j),print([[j,prime(j)],NP(j,1)]),
      print([[j,prime(j)],NP(j,0)]))))))
);
}
\end{verbatim}
}
\vspace*{3mm}

We postpone displaying the result of {\tt NP(j,GRH)} for primes 
$p_j< 20000$ $(j\leq 2262)$ in PARI/GP to Section \ref{seResult}. 
The result of {\tt NP(j,GRH)} for\\
{\rm (i)} $1\leq j\leq 500$ ($2\leq p_j\leq 3571)$;\\
{\rm (ii)} $501\leq j\leq 1000$ ($3581\leq p_j\leq 7919)$;\\
{\rm (iii)} $1001\leq j\leq 1500$ ($7927\leq p_j\leq 12533)$;\\
{\rm (iv)} $1501\leq j\leq 2000$ ($12569\leq p_j\leq 17389)$;\\
{\rm (v)} $2001\leq j\leq 2262$ ($17393\leq p_j\leq 19997)$\\
will be given in Section \ref{seResult} respectively.

\begin{proof}[Proof of Theorem \ref{thmain}]
Let $p<20000$ be a prime. 
Theorem \ref{thmain} follows from the result in Subsection \ref{ss31} 
(resp.  Subsection \ref{ss32},  Subsection \ref{ss33}) 
for $p\in S_0\cup S_1$ (resp. $p\in T_0\cup T_1$, 
$p\not\in U\cup S_0\cup S_1\cup T_0\cup T_1$). 
\end{proof}

%\onecolumn 

\begin{aremark}
From the view point of Theorems \ref{thSwan} and  \ref{thVos}, 
Noether's problem for $C_p$ over $\bQ$ is closely related to 
Weber's class number problem (see e.g. Fukuda and Komtsu 
\cite{FK09}, \cite{FK10}, \cite{FK11}). 
Actually, after this paper was posted on the arXiv, 
Fukuda \cite{Fuk14} announced to the author 
that he proved the non-rationality of $\bQ(C_{59})$ over $\bQ$ 
without the GRH. 
Independently, Lawrence C. Washington pointed out to 
John C. Miller that his methods for finding principal ideals 
of real cyclotomic fields in \cite{Mil1}, \cite{Mil2} 
may be valid for $\bQ(\zeta_{p-1})$ at least some small primes $p$. 
Indeed, Miller \cite{Mil14} proved that 
$\bQ(C_p)$ is not rational over $\bQ$ 
for $p=59$ (resp. $251$) 
without the GRH (resp. under the GRH) 
by using a similar technique as in \cite{Mil1}, \cite{Mil2}. 
It should be interesting how to improve the methods 
of Fukuda and Miller for higher primes $p$. 
\end{aremark}

%%%%%%%%%%%%%%%%%%%%%%%%%%%%%%%%%%%%%%%%%%%%%%%%%%%%%%%%%%%%%%%%%%%%%%%%%%%%%%%%%%%%%%%%%%%
\newpage

\section{Rsult of {\tt NP(j,GRH)} for $p_j<20000$ $(1\leq j\leq 2262)$}
\label{seResult}

%\newpage
%{\small
%\begin{multicols}{3}
{\tiny
\begin{multicols}{4}
\begin{verbatim}
gp > #
   timer = 1 (on)
gp > allocatemem(2^4*10^8)
  ***   Warning: new 
stack size = 1600000000 
(1525.879 Mbytes).
gp > NPs(1,500)
[[1, 2], [Ratioal, 0]]
[[2, 3], [Ratioal, 0]]
[[3, 5], [Ratioal, 0]]
[[4, 7], [Ratioal, 0]]
[[5, 11], [Ratioal, 0]]
[[6, 13], [Ratioal, 0]]
[[7, 17], [Ratioal, 0]]
[[8, 19], [Ratioal, 0]]
[[9, 23], [Ratioal, 0]]
[[10, 29], [Ratioal, 0]]
[[11, 31], [Ratioal, 0]]
[[12, 37], [Ratioal, 0]]
[[13, 41], [Ratioal, 0]]
[[14, 43], [Ratioal, 0]]
[[15, 47], [2, 2, 0]]
[[16, 53], [6, 2, 0]]
[[17, 59], [28, 4, 1]]
[[18, 61], [Ratioal, 0]]
[[19, 67], [Ratioal, 0]]
[[20, 71], [Ratioal, 0]]
[[21, 73], [6, 2, 0]]
[[22, 79], [2, 2, 0]]
[[23, 83], [40, 8, 1]]
[[24, 89], [10, 2, 0]]
[[25, 97], [8, 2, 0]]
[[26, 101], [10, 2, 0]]
[[27, 103], [4, 2, 0]]
[[28, 107], [52, 4, 1]]
[[29, 109], [18, 2, 0]]
[[30, 113], [2, 2, 0]]
[[31, 127], [6, 2, 0]]
[[32, 131], [12, 4, 0]]
[[33, 137], [2, 2, 0]]
[[34, 139], [2, 2, 0]]
[[35, 149], [6, 2, 0]]
[[36, 151], [10, 2, 0]]
[[37, 157], [6, 2, 0]]
[[38, 163], [54, 2, 1]]
[[39, 167], [2, 2, 0]]
[[40, 173], [6, 2, 0]]
[[41, 179], [8, 8, 0]]
[[42, 181], [6, 2, 0]]
[[43, 191], [2, 2, 0]]
[[44, 193], [8, 2, 0]]
[[45, 197], "T0"]
[[46, 199], [6, 2, 0]]
[[47, 211], [12, 2, 0]]
[[48, 223], [2, 2, 0]]
[[49, 227], "T0"]
[[50, 229], [6, 2, 0]]
[[51, 233], [2, 2, 0]]
[[52, 239], [2, 2, 0]]
[[53, 241], [4, 2, 0]]
[[54, 251], "?"]
[[55, 257], [16, 2, 0]]
[[56, 263], [2, 2, 0]]
[[57, 269], [6, 2, 0]]
[[58, 271], [18, 2, 0]]
[[59, 277], [2, 2, 0]]
[[60, 281], [4, 2, 0]]
[[61, 283], [2, 2, 0]]
[[62, 293], [4, 2, 0]]
[[63, 307], [4, 2, 0]]
[[64, 311], [2, 2, 0]]
[[65, 313], [2, 2, 0]]
[[66, 317], [2, 2, 0]]
[[67, 331], [2, 2, 0]]
[[68, 337], [2, 2, 0]]
[[69, 347], "?"]
[[70, 349], [2, 2, 0]]
[[71, 353], [4, 2, 0]]
[[72, 359], [2, 2, 0]]
[[73, 367], [2, 2, 0]]
[[74, 373], [2, 2, 0]]
[[75, 379], [6, 2, 0]]
[[76, 383], [2, 2, 0]]
[[77, 389], [4, 2, 0]]
[[78, 397], [6, 2, 0]]
[[79, 401], [4, 2, 0]]
[[80, 409], [2, 2, 0]]
[[81, 419], [6, 2, 0]]
[[82, 421], [4, 2, 0]]
[[83, 431], [2, 2, 0]]
[[84, 433], [6, 2, 0]]
[[85, 439], [2, 2, 0]]
[[86, 443], [2, 4, 0]]
[[87, 449], [8, 2, 0]]
[[88, 457], [2, 2, 0]]
[[89, 461], [2, 2, 0]]
[[90, 463], [2, 2, 0]]
[[91, 467], [8, 8, 0]]
[[92, 479], [2, 2, 0]]
[[93, 487], [54, 2, 1]]
[[94, 491], "T0"]
[[95, 499], [2, 2, 0]]
[[96, 503], [2, 2, 0]]
[[97, 509], [6, 2, 0]]
[[98, 521], [2, 2, 0]]
[[99, 523], [2, 2, 0]]
[[100, 541], [6, 2, 0]]
[[101, 547], [6, 2, 0]]
[[102, 557], [6, 2, 0]]
[[103, 563], [8, 8, 0]]
[[104, 569], [2, 2, 0]]
[[105, 571], [2, 2, 0]]
[[106, 577], [6, 2, 0]]
[[107, 587], "?"]
[[108, 593], [2, 2, 0]]
[[109, 599], [2, 2, 0]]
[[110, 601], [10, 2, 0]]
[[111, 607], [2, 2, 0]]
[[112, 613], [4, 2, 0]]
[[113, 617], [2, 2, 0]]
[[114, 619], [2, 2, 0]]
[[115, 631], [6, 2, 0]]
[[116, 641], [4, 2, 0]]
[[117, 643], [2, 2, 0]]
[[118, 647], [4, 2, 0]]
[[119, 653], [6, 2, 0]]
[[120, 659], [2, 2, 0]]
[[121, 661], [4, 2, 0]]
[[122, 673], [4, 2, 0]]
[[123, 677], [26, 2, 1]]
[[124, 683], [2, 2, 0]]
[[125, 691], [4, 2, 0]]
[[126, 701], [6, 2, 0]]
[[127, 709], [2, 2, 0]]
[[128, 719], [2, 2, 0]]
[[129, 727], [22, 2, 1]]
[[130, 733], [2, 2, 0]]
[[131, 739], [4, 2, 0]]
[[132, 743], [6, 2, 0]]
[[133, 751], [10, 2, 0]]
[[134, 757], [6, 2, 0]]
[[135, 761], [2, 2, 0]]
[[136, 769], [8, 2, 0]]
[[137, 773], [4, 2, 0]]
[[138, 787], [2, 2, 0]]
[[139, 797], [6, 2, 0]]
[[140, 809], [2, 2, 0]]
[[141, 811], [18, 2, 0]]
[[142, 821], [2, 2, 0]]
[[143, 823], [2, 2, 0]]
[[144, 827], [6, 2, 0]]
[[145, 829], [2, 2, 0]]
[[146, 839], [2, 2, 0]]
[[147, 853], [2, 2, 0]]
[[148, 857], [2, 2, 0]]
[[149, 859], [2, 2, 0]]
[[150, 863], [2, 2, 0]]
[[151, 877], [2, 2, 0]]
[[152, 881], [4, 2, 0]]
[[153, 883], [6, 2, 0]]
[[154, 887], [2, 2, 0]]
[[155, 907], [2, 2, 0]]
[[156, 911], [2, 2, 0]]
[[157, 919], [4, 2, 0]]
[[158, 929], [4, 2, 0]]
[[159, 937], [2, 2, 0]]
[[160, 941], [2, 2, 0]]
[[161, 947], [2, 2, 0]]
[[162, 953], [2, 2, 0]]
[[163, 967], [2, 2, 0]]
[[164, 971], [4, 4, 0]]
[[165, 977], [2, 2, 0]]
[[166, 983], [2, 2, 0]]
[[167, 991], [2, 2, 0]]
[[168, 997], [2, 2, 0]]
[[169, 1009], [2, 2, 0]]
[[170, 1013], [2, 2, 0]]
[[171, 1019], [4, 4, 0]]
[[172, 1021], [2, 2, 0]]
[[173, 1031], [2, 2, 0]]
[[174, 1033], [2, 2, 0]]
[[175, 1039], [2, 2, 0]]
[[176, 1049], [2, 2, 0]]
[[177, 1051], [10, 2, 0]]
[[178, 1061], [2, 2, 0]]
[[179, 1063], [6, 2, 0]]
[[180, 1069], [2, 2, 0]]
[[181, 1087], [2, 2, 0]]
[[182, 1091], [2, 4, 0]]
[[183, 1093], [2, 2, 0]]
[[184, 1097], [2, 2, 0]]
[[185, 1103], [2, 2, 0]]
[[186, 1109], [3, 2, 0]]
[[187, 1117], [6, 2, 0]]
[[188, 1123], [4, 2, 0]]
[[189, 1129], [2, 2, 0]]
[[190, 1151], [10, 2, 0]]
[[191, 1153], [6, 2, 0]]
[[192, 1163], [2, 2, 0]]
[[193, 1171], [4, 2, 0]]
[[194, 1181], [4, 2, 0]]
[[195, 1187], [16, 16, 1]]
[[196, 1193], [2, 2, 0]]
[[197, 1201], [4, 2, 0]]
[[198, 1213], [2, 2, 0]]
[[199, 1217], [2, 2, 0]]
[[200, 1223], [2, 2, 0]]
[[201, 1229], [6, 2, 0]]
[[202, 1231], [2, 2, 0]]
[[203, 1237], [2, 2, 0]]
[[204, 1249], [2, 2, 0]]
[[205, 1259], [4, 4, 0]]
[[206, 1277], [2, 2, 0]]
[[207, 1279], [2, 2, 0]]
[[208, 1283], [4, 4, 0]]
[[209, 1289], [2, 2, 0]]
[[210, 1291], [2, 2, 0]]
[[211, 1297], [6, 2, 0]]
[[212, 1301], [2, 2, 0]]
[[213, 1303], [2, 2, 0]]
[[214, 1307], [4, 4, 0]]
[[215, 1319], [2, 2, 0]]
[[216, 1321], [2, 2, 0]]
[[217, 1327], [2, 2, 0]]
[[218, 1361], [2, 2, 0]]
[[219, 1367], [2, 2, 0]]
[[220, 1373], "T0"]
[[221, 1381], [2, 2, 0]]
[[222, 1399], [2, 2, 0]]
[[223, 1409], [8, 2, 0]]
[[224, 1423], [2, 2, 0]]
[[225, 1427], [2, 2, 0]]
[[226, 1429], [2, 2, 0]]
[[227, 1433], [2, 2, 0]]
[[228, 1439], [2, 2, 0]]
[[229, 1447], [2, 2, 0]]
[[230, 1451], [2, 2, 0]]
[[231, 1453], [10, 2, 0]]
[[232, 1459], [54, 2, 1]]
[[233, 1471], [12, 2, 0]]
[[234, 1481], [2, 2, 0]]
[[235, 1483], [2, 2, 0]]
[[236, 1487], [2, 2, 0]]
[[237, 1489], [2, 2, 0]]
[[238, 1493], [4, 2, 0]]
[[239, 1499], [2, 2, 0]]
[[240, 1511], [2, 2, 0]]
[[241, 1523], "T0"]
[[242, 1531], [2, 2, 0]]
[[243, 1543], [2, 2, 0]]
[[244, 1549], [2, 2, 0]]
[[245, 1553], [2, 2, 0]]
[[246, 1559], [2, 2, 0]]
[[247, 1567], [6, 2, 0]]
[[248, 1571], [2, 2, 0]]
[[249, 1579], [2, 2, 0]]
[[250, 1583], [2, 2, 0]]
[[251, 1597], [2, 2, 0]]
[[252, 1601], [4, 2, 0]]
[[253, 1607], [4, 2, 0]]
[[254, 1609], [2, 2, 0]]
[[255, 1613], [2, 2, 0]]
[[256, 1619], "T0"]
[[257, 1621], [12, 2, 0]]
[[258, 1627], [2, 2, 0]]
[[259, 1637], [4, 2, 0]]
[[260, 1657], [2, 2, 0]]
[[261, 1663], [2, 2, 0]]
[[262, 1667], [2, 2, 0]]
[[263, 1669], [2, 2, 0]]
[[264, 1693], [2, 2, 0]]
[[265, 1697], [2, 2, 0]]
[[266, 1699], [2, 2, 0]]
[[267, 1709], [2, 2, 0]]
[[268, 1721], [2, 2, 0]]
[[269, 1723], [2, 2, 0]]
[[270, 1733], [4, 2, 0]]
[[271, 1741], [2, 2, 0]]
[[272, 1747], [2, 2, 0]]
[[273, 1753], [2, 2, 0]]
[[274, 1759], [2, 2, 0]]
[[275, 1777], [2, 2, 0]]
[[276, 1783], "T0"]
[[277, 1787], [2, 2, 0]]
[[278, 1789], [2, 2, 0]]
[[279, 1801], [4, 2, 0]]
[[280, 1811], [2, 2, 0]]
[[281, 1823], [2, 2, 0]]
[[282, 1831], [2, 2, 0]]
[[283, 1847], [2, 2, 0]]
[[284, 1861], [2, 2, 0]]
[[285, 1867], [2, 2, 0]]
[[286, 1871], [2, 2, 0]]
[[287, 1873], [2, 2, 0]]
[[288, 1877], [2, 2, 0]]
[[289, 1879], [2, 2, 0]]
[[290, 1889], [2, 2, 0]]
[[291, 1901], [6, 2, 0]]
[[292, 1907], [7, 7, 0]]
[[293, 1913], [2, 2, 0]]
[[294, 1931], [4, 4, 0]]
[[295, 1933], [2, 2, 0]]
[[296, 1949], [6, 2, 0]]
[[297, 1951], [4, 2, 0]]
[[298, 1973], [2, 2, 0]]
[[299, 1979], [2, 2, 0]]
[[300, 1987], [2, 2, 0]]
[[301, 1993], [2, 2, 0]]
[[302, 1997], [2, 2, 0]]
[[303, 1999], [2, 2, 0]]
[[304, 2003], [2, 2, 0]]
[[305, 2011], [2, 2, 0]]
[[306, 2017], [4, 2, 0]]
[[307, 2027], [4, 4, 0]]
[[308, 2029], [6, 2, 0]]
[[309, 2039], [2, 2, 0]]
[[310, 2053], [6, 2, 0]]
[[311, 2063], [2, 2, 0]]
[[312, 2069], [2, 2, 0]]
[[313, 2081], [2, 2, 0]]
[[314, 2083], [2, 2, 0]]
[[315, 2087], [2, 2, 0]]
[[316, 2089], [2, 2, 0]]
[[317, 2099], "T0"]
[[318, 2111], [2, 2, 0]]
[[319, 2113], [4, 2, 0]]
[[320, 2129], [2, 2, 0]]
[[321, 2131], [2, 2, 0]]
[[322, 2137], [2, 2, 0]]
[[323, 2141], [2, 2, 0]]
[[324, 2143], [2, 2, 0]]
[[325, 2153], [2, 2, 0]]
[[326, 2161], [4, 2, 0]]
[[327, 2179], [6, 2, 0]]
[[328, 2203], [2, 2, 0]]
[[329, 2207], [2, 2, 0]]
[[330, 2213], [2, 2, 0]]
[[331, 2221], [2, 2, 0]]
[[332, 2237], [4, 2, 0]]
[[333, 2239], [2, 2, 0]]
[[334, 2243], [2, 2, 0]]
[[335, 2251], [10, 2, 0]]
[[336, 2267], [2, 2, 0]]
[[337, 2269], [6, 2, 0]]
[[338, 2273], [2, 2, 0]]
[[339, 2281], [2, 2, 0]]
[[340, 2287], [2, 2, 0]]
[[341, 2293], [2, 2, 0]]
[[342, 2297], [2, 2, 0]]
[[343, 2309], [2, 2, 0]]
[[344, 2311], [2, 2, 0]]
[[345, 2333], [2, 2, 0]]
[[346, 2339], [2, 2, 0]]
[[347, 2341], [2, 2, 0]]
[[348, 2347], [2, 2, 0]]
[[349, 2351], [10, 2, 0]]
[[350, 2357], [2, 2, 0]]
[[351, 2371], [2, 2, 0]]
[[352, 2377], [4, 2, 0]]
[[353, 2381], [2, 2, 0]]
[[354, 2383], [2, 2, 0]]
[[355, 2389], [2, 2, 0]]
[[356, 2393], [2, 2, 0]]
[[357, 2399], [2, 2, 0]]
[[358, 2411], [2, 4, 0]]
[[359, 2417], [2, 2, 0]]
[[360, 2423], [2, 2, 0]]
[[361, 2437], [2, 2, 0]]
[[362, 2441], [2, 2, 0]]
[[363, 2447], [2, 2, 0]]
[[364, 2459], "?"]
[[365, 2467], [2, 2, 0]]
[[366, 2473], [2, 2, 0]]
[[367, 2477], [6, 2, 0]]
[[368, 2503], [2, 2, 0]]
[[369, 2521], [2, 2, 0]]
[[370, 2531], [2, 2, 0]]
[[371, 2539], [6, 2, 0]]
[[372, 2543], [2, 2, 0]]
[[373, 2549], [4, 2, 0]]
[[374, 2551], [4, 2, 0]]
[[375, 2557], [2, 2, 0]]
[[376, 2579], "T0"]
[[377, 2591], [2, 2, 0]]
[[378, 2593], [8, 2, 0]]
[[379, 2609], [2, 2, 0]]
[[380, 2617], [2, 2, 0]]
[[381, 2621], [2, 2, 0]]
[[382, 2633], [2, 2, 0]]
[[383, 2647], [6, 2, 0]]
[[384, 2657], [4, 2, 0]]
[[385, 2659], [2, 2, 0]]
[[386, 2663], [22, 2, 1]]
[[387, 2671], [2, 2, 0]]
[[388, 2677], [2, 2, 0]]
[[389, 2683], [2, 2, 0]]
[[390, 2687], [2, 2, 0]]
[[391, 2689], [2, 2, 0]]
[[392, 2693], [4, 2, 0]]
[[393, 2699], [2, 2, 0]]
[[394, 2707], [2, 2, 0]]
[[395, 2711], [2, 2, 0]]
[[396, 2713], [2, 2, 0]]
[[397, 2719], [2, 2, 0]]
[[398, 2729], [2, 2, 0]]
[[399, 2731], [2, 2, 0]]
[[400, 2741], [2, 2, 0]]
[[401, 2749], [2, 2, 0]]
[[402, 2753], [2, 2, 0]]
[[403, 2767], [2, 2, 0]]
[[404, 2777], [2, 2, 0]]
[[405, 2789], [2, 2, 0]]
[[406, 2791], [2, 2, 0]]
[[407, 2797], [2, 2, 0]]
[[408, 2801], [2, 2, 0]]
[[409, 2803], [2, 2, 0]]
[[410, 2819], "?"]
[[411, 2833], [2, 2, 0]]
[[412, 2837], [4, 2, 0]]
[[413, 2843], [2, 2, 0]]
[[414, 2851], [2, 2, 0]]
[[415, 2857], [2, 2, 0]]
[[416, 2861], [2, 2, 0]]
[[417, 2879], [2, 2, 0]]
[[418, 2887], [2, 2, 0]]
[[419, 2897], [2, 2, 0]]
[[420, 2903], [2, 2, 0]]
[[421, 2909], [2, 2, 0]]
[[422, 2917], [18, 2, 0]]
[[423, 2927], [2, 2, 0]]
[[424, 2939], [4, 2, 0]]
[[425, 2953], [2, 2, 0]]
[[426, 2957], [6, 2, 0]]
[[427, 2963], "T0"]
[[428, 2969], [2, 2, 0]]
[[429, 2971], [6, 2, 0]]
[[430, 2999], [2, 2, 0]]
[[431, 3001], [10, 2, 0]]
[[432, 3011], [2, 2, 0]]
[[433, 3019], [2, 2, 0]]
[[434, 3023], [2, 2, 0]]
[[435, 3037], [2, 2, 0]]
[[436, 3041], [2, 2, 0]]
[[437, 3049], [2, 2, 0]]
[[438, 3061], [2, 2, 0]]
[[439, 3067], [2, 2, 0]]
[[440, 3079], [6, 2, 0]]
[[441, 3083], [2, 2, 0]]
[[442, 3089], [2, 2, 0]]
[[443, 3109], [2, 2, 0]]
[[444, 3119], [2, 2, 0]]
[[445, 3121], [2, 2, 0]]
[[446, 3137], [2, 2, 0]]
[[447, 3163], [2, 2, 0]]
[[448, 3167], [2, 2, 0]]
[[449, 3169], [2, 2, 0]]
[[450, 3181], [2, 2, 0]]
[[451, 3187], [2, 2, 0]]
[[452, 3191], [2, 2, 0]]
[[453, 3203], [2, 2, 0]]
[[454, 3209], [2, 2, 0]]
[[455, 3217], [2, 2, 0]]
[[456, 3221], [2, 2, 0]]
[[457, 3229], [2, 2, 0]]
[[458, 3251], [12, 4, 0]]
[[459, 3253], [2, 2, 0]]
[[460, 3257], [2, 2, 0]]
[[461, 3259], [2, 2, 0]]
[[462, 3271], [2, 2, 0]]
[[463, 3299], "?"]
[[464, 3301], [2, 2, 0]]
[[465, 3307], [2, 2, 0]]
[[466, 3313], [2, 2, 0]]
[[467, 3319], [2, 2, 0]]
[[468, 3323], [2, 2, 0]]
[[469, 3329], [4, 2, 0]]
[[470, 3331], [2, 2, 0]]
[[471, 3343], [2, 2, 0]]
[[472, 3347], [2, 2, 0]]
[[473, 3359], [2, 2, 0]]
[[474, 3361], [2, 2, 0]]
[[475, 3371], [4, 4, 0]]
[[476, 3373], [2, 2, 0]]
[[477, 3389], [2, 2, 0]]
[[478, 3391], [2, 2, 0]]
[[479, 3407], [2, 2, 0]]
[[480, 3413], [3, 2, 0]]
[[481, 3433], [2, 2, 0]]
[[482, 3449], [2, 2, 0]]
[[483, 3457], [8, 2, 0]]
[[484, 3461], [2, 2, 0]]
[[485, 3463], [2, 2, 0]]
[[486, 3467], [4, 4, 0]]
[[487, 3469], [2, 2, 0]]
[[488, 3491], [3, 3, 0]]
[[489, 3499], [2, 2, 0]]
[[490, 3511], [4, 2, 0]]
[[491, 3517], [2, 2, 0]]
[[492, 3527], [4, 2, 0]]
[[493, 3529], [6, 2, 0]]
[[494, 3533], [6, 2, 0]]
[[495, 3539], [4, 4, 0]]
[[496, 3541], [2, 2, 0]]
[[497, 3547], [2, 2, 0]]
[[498, 3557], [2, 2, 0]]
[[499, 3559], [2, 2, 0]]
[[500, 3571], [2, 2, 0]]
time = 41min, 25,327 ms.
\end{verbatim}
\end{multicols}
}

\newpage
{\tiny
\begin{multicols}{4}
\begin{verbatim}
gp > NPs(501,1000)
[[501, 3581], [2, 2, 0]]
[[502, 3583], [2, 2, 0]]
[[503, 3593], [2, 2, 0]]
[[504, 3607], [2, 2, 0]]
[[505, 3613], [2, 2, 0]]
[[506, 3617], [2, 2, 0]]
[[507, 3623], [2, 2, 0]]
[[508, 3631], [8, 2, 0]]
[[509, 3637], [2, 2, 0]]
[[510, 3643], [2, 2, 0]]
[[511, 3659], [2, 2, 0]]
[[512, 3671], [2, 2, 0]]
[[513, 3673], [2, 2, 0]]
[[514, 3677], [6, 2, 0]]
[[515, 3691], [2, 2, 0]]
[[516, 3697], [2, 2, 0]]
[[517, 3701], [2, 2, 0]]
[[518, 3709], [2, 2, 0]]
[[519, 3719], [6, 2, 0]]
[[520, 3727], [2, 2, 0]]
[[521, 3733], [2, 2, 0]]
[[522, 3739], [2, 2, 0]]
[[523, 3761], [2, 2, 0]]
[[524, 3767], [2, 2, 0]]
[[525, 3769], [2, 2, 0]]
[[526, 3779], [16, 16, 1]]
[[527, 3793], [2, 2, 0]]
[[528, 3797], [2, 2, 0]]
[[529, 3803], [2, 2, 0]]
[[530, 3821], [2, 2, 0]]
[[531, 3823], [2, 2, 0]]
[[532, 3833], [2, 2, 0]]
[[533, 3847], [2, 2, 0]]
[[534, 3851], [6, 2, 0]]
[[535, 3853], [2, 2, 0]]
[[536, 3863], [2, 2, 0]]
[[537, 3877], [2, 2, 0]]
[[538, 3881], [2, 2, 0]]
[[539, 3889], [6, 2, 0]]
[[540, 3907], [2, 2, 0]]
[[541, 3911], [2, 2, 0]]
[[542, 3917], [2, 2, 0]]
[[543, 3919], [2, 2, 0]]
[[544, 3923], [4, 2, 0]]
[[545, 3929], [2, 2, 0]]
[[546, 3931], [2, 2, 0]]
[[547, 3943], [3, 2, 0]]
[[548, 3947], [4, 4, 0]]
[[549, 3967], [2, 2, 0]]
[[550, 3989], [4, 2, 0]]
[[551, 4001], [4, 2, 0]]
[[552, 4003], [2, 2, 0]]
[[553, 4007], [2, 2, 0]]
[[554, 4013], [2, 2, 0]]
[[555, 4019], [2, 2, 0]]
[[556, 4021], [2, 2, 0]]
[[557, 4027], [2, 2, 0]]
[[558, 4049], [2, 2, 0]]
[[559, 4051], [10, 2, 0]]
[[560, 4057], [2, 2, 0]]
[[561, 4073], [2, 2, 0]]
[[562, 4079], [2, 2, 0]]
[[563, 4091], [2, 4, 0]]
[[564, 4093], [2, 2, 0]]
[[565, 4099], [2, 2, 0]]
[[566, 4111], [2, 2, 0]]
[[567, 4127], [2, 2, 0]]
[[568, 4129], [2, 2, 0]]
[[569, 4133], [4, 2, 0]]
[[570, 4139], [4, 4, 0]]
[[571, 4153], [2, 2, 0]]
[[572, 4157], [6, 2, 0]]
[[573, 4159], [2, 2, 0]]
[[574, 4177], [2, 2, 0]]
[[575, 4201], [2, 2, 0]]
[[576, 4211], [4, 2, 0]]
[[577, 4217], [2, 2, 0]]
[[578, 4219], [2, 2, 0]]
[[579, 4229], [2, 2, 0]]
[[580, 4231], [2, 2, 0]]
[[581, 4241], [2, 2, 0]]
[[582, 4243], [2, 2, 0]]
[[583, 4253], [6, 2, 0]]
[[584, 4259], [16, 16, 1]]
[[585, 4261], [2, 2, 0]]
[[586, 4271], [2, 2, 0]]
[[587, 4273], [2, 2, 0]]
[[588, 4283], [4, 4, 0]]
[[589, 4289], [2, 2, 0]]
[[590, 4297], [2, 2, 0]]
[[591, 4327], [2, 2, 0]]
[[592, 4337], [2, 2, 0]]
[[593, 4339], [2, 2, 0]]
[[594, 4349], [2, 2, 0]]
[[595, 4357], [6, 2, 0]]
[[596, 4363], [2, 2, 0]]
[[597, 4373], [2, 2, 0]]
[[598, 4391], [2, 2, 0]]
[[599, 4397], [2, 2, 0]]
[[600, 4409], [2, 2, 0]]
[[601, 4421], [2, 2, 0]]
[[602, 4423], [2, 2, 0]]
[[603, 4441], [2, 2, 0]]
[[604, 4447], [2, 2, 0]]
[[605, 4451], [2, 2, 0]]
[[606, 4457], [2, 2, 0]]
[[607, 4463], [2, 2, 0]]
[[608, 4481], [2, 2, 0]]
[[609, 4483], [2, 2, 0]]
[[610, 4493], [6, 2, 0]]
[[611, 4507], [2, 2, 0]]
[[612, 4513], [2, 2, 0]]
[[613, 4517], [2, 2, 0]]
[[614, 4519], [2, 2, 0]]
[[615, 4523], [2, 2, 0]]
[[616, 4547], "?"]
[[617, 4549], [2, 2, 0]]
[[618, 4561], [2, 2, 0]]
[[619, 4567], [2, 2, 0]]
[[620, 4583], [2, 2, 0]]
[[621, 4591], [2, 2, 0]]
[[622, 4597], [2, 2, 0]]
[[623, 4603], [2, 2, 0]]
[[624, 4621], [2, 2, 0]]
[[625, 4637], [2, 2, 0]]
[[626, 4639], [2, 2, 0]]
[[627, 4643], [2, 2, 0]]
[[628, 4649], [2, 2, 0]]
[[629, 4651], [2, 2, 0]]
[[630, 4657], [2, 2, 0]]
[[631, 4663], [2, 2, 0]]
[[632, 4673], [2, 2, 0]]
[[633, 4679], [2, 2, 0]]
[[634, 4691], [2, 2, 0]]
[[635, 4703], [2, 2, 0]]
[[636, 4721], [2, 2, 0]]
[[637, 4723], [2, 2, 0]]
[[638, 4729], [2, 2, 0]]
[[639, 4733], [4, 2, 0]]
[[640, 4751], [2, 2, 0]]
[[641, 4759], [2, 2, 0]]
[[642, 4783], [2, 2, 0]]
[[643, 4787], "?"]
[[644, 4789], [2, 2, 0]]
[[645, 4793], [2, 2, 0]]
[[646, 4799], [2, 2, 0]]
[[647, 4801], [4, 2, 0]]
[[648, 4813], [2, 2, 0]]
[[649, 4817], [2, 2, 0]]
[[650, 4831], [2, 2, 0]]
[[651, 4861], [6, 2, 0]]
[[652, 4871], [2, 2, 0]]
[[653, 4877], [2, 2, 0]]
[[654, 4889], [2, 2, 0]]
[[655, 4903], [2, 2, 0]]
[[656, 4909], [2, 2, 0]]
[[657, 4919], [2, 2, 0]]
[[658, 4931], [4, 4, 0]]
[[659, 4933], [4, 2, 0]]
[[660, 4937], [2, 2, 0]]
[[661, 4943], [2, 2, 0]]
[[662, 4951], [2, 2, 0]]
[[663, 4957], [2, 2, 0]]
[[664, 4967], [2, 2, 0]]
[[665, 4969], [2, 2, 0]]
[[666, 4973], [2, 2, 0]]
[[667, 4987], [2, 2, 0]]
[[668, 4993], [2, 2, 0]]
[[669, 4999], [2, 2, 0]]
[[670, 5003], [2, 2, 0]]
[[671, 5009], [2, 2, 0]]
[[672, 5011], [2, 2, 0]]
[[673, 5021], [2, 2, 0]]
[[674, 5023], [2, 2, 0]]
[[675, 5039], [2, 2, 0]]
[[676, 5051], [2, 2, 0]]
[[677, 5059], [2, 2, 0]]
[[678, 5077], [2, 2, 0]]
[[679, 5081], [2, 2, 0]]
[[680, 5087], [2, 2, 0]]
[[681, 5099], [4, 4, 0]]
[[682, 5101], [2, 2, 0]]
[[683, 5107], [2, 2, 0]]
[[684, 5113], [2, 2, 0]]
[[685, 5119], [2, 2, 0]]
[[686, 5147], [2, 2, 0]]
[[687, 5153], [2, 2, 0]]
[[688, 5167], [2, 2, 0]]
[[689, 5171], [2, 2, 0]]
[[690, 5179], [2, 2, 0]]
[[691, 5189], [2, 2, 0]]
[[692, 5197], [2, 2, 0]]
[[693, 5209], [2, 2, 0]]
[[694, 5227], [2, 2, 0]]
[[695, 5231], [2, 2, 0]]
[[696, 5233], [2, 2, 0]]
[[697, 5237], [2, 2, 0]]
[[698, 5261], [2, 2, 0]]
[[699, 5273], [2, 2, 0]]
[[700, 5279], [2, 2, 0]]
[[701, 5281], [2, 2, 0]]
[[702, 5297], [2, 2, 0]]
[[703, 5303], [2, 2, 0]]
[[704, 5309], [2, 2, 0]]
[[705, 5323], [2, 2, 0]]
[[706, 5333], [2, 2, 0]]
[[707, 5347], [6, 2, 0]]
[[708, 5351], [2, 2, 0]]
[[709, 5381], [2, 2, 0]]
[[710, 5387], [4, 4, 0]]
[[711, 5393], [2, 2, 0]]
[[712, 5399], [2, 2, 0]]
[[713, 5407], [2, 2, 0]]
[[714, 5413], [2, 2, 0]]
[[715, 5417], [2, 2, 0]]
[[716, 5419], [2, 2, 0]]
[[717, 5431], [2, 2, 0]]
[[718, 5437], [2, 2, 0]]
[[719, 5441], [2, 2, 0]]
[[720, 5443], [2, 2, 0]]
[[721, 5449], [2, 2, 0]]
[[722, 5471], [2, 2, 0]]
[[723, 5477], [6, 2, 0]]
[[724, 5479], [2, 2, 0]]
[[725, 5483], [4, 4, 0]]
[[726, 5501], [10, 2, 0]]
[[727, 5503], [2, 2, 0]]
[[728, 5507], "T0"]
[[729, 5519], [2, 2, 0]]
[[730, 5521], [2, 2, 0]]
[[731, 5527], [2, 2, 0]]
[[732, 5531], [2, 2, 0]]
[[733, 5557], [2, 2, 0]]
[[734, 5563], [2, 2, 0]]
[[735, 5569], [2, 2, 0]]
[[736, 5573], [2, 2, 0]]
[[737, 5581], [2, 2, 0]]
[[738, 5591], [2, 2, 0]]
[[739, 5623], [2, 2, 0]]
[[740, 5639], [2, 2, 0]]
[[741, 5641], [2, 2, 0]]
[[742, 5647], [2, 2, 0]]
[[743, 5651], [4, 4, 0]]
[[744, 5653], [2, 2, 0]]
[[745, 5657], [2, 2, 0]]
[[746, 5659], [2, 2, 0]]
[[747, 5669], [2, 2, 0]]
[[748, 5683], [2, 2, 0]]
[[749, 5689], [2, 2, 0]]
[[750, 5693], [6, 2, 0]]
[[751, 5701], [2, 2, 0]]
[[752, 5711], [2, 2, 0]]
[[753, 5717], [2, 2, 0]]
[[754, 5737], [2, 2, 0]]
[[755, 5741], [2, 2, 0]]
[[756, 5743], [2, 2, 0]]
[[757, 5749], [2, 2, 0]]
[[758, 5779], [2, 2, 0]]
[[759, 5783], [2, 2, 0]]
[[760, 5791], [2, 2, 0]]
[[761, 5801], [2, 2, 0]]
[[762, 5807], [2, 2, 0]]
[[763, 5813], [4, 2, 0]]
[[764, 5821], [2, 2, 0]]
[[765, 5827], [2, 2, 0]]
[[766, 5839], [2, 2, 0]]
[[767, 5843], [2, 2, 0]]
[[768, 5849], [2, 2, 0]]
[[769, 5851], [2, 2, 0]]
[[770, 5857], [2, 2, 0]]
[[771, 5861], [2, 2, 0]]
[[772, 5867], [2, 2, 0]]
[[773, 5869], [2, 2, 0]]
[[774, 5879], [2, 2, 0]]
[[775, 5881], [2, 2, 0]]
[[776, 5897], [2, 2, 0]]
[[777, 5903], [2, 2, 0]]
[[778, 5923], [2, 2, 0]]
[[779, 5927], [2, 2, 0]]
[[780, 5939], "T0"]
[[781, 5953], [2, 2, 0]]
[[782, 5981], [2, 2, 0]]
[[783, 5987], "S0"]
[[784, 6007], [2, 2, 0]]
[[785, 6011], [2, 4, 0]]
[[786, 6029], [2, 2, 0]]
[[787, 6037], [2, 2, 0]]
[[788, 6043], [2, 2, 0]]
[[789, 6047], [2, 2, 0]]
[[790, 6053], [2, 2, 0]]
[[791, 6067], [2, 2, 0]]
[[792, 6073], [2, 2, 0]]
[[793, 6079], [2, 2, 0]]
[[794, 6089], [2, 2, 0]]
[[795, 6091], [2, 2, 0]]
[[796, 6101], [2, 2, 0]]
[[797, 6113], [2, 2, 0]]
[[798, 6121], [2, 2, 0]]
[[799, 6131], [4, 4, 0]]
[[800, 6133], [2, 2, 0]]
[[801, 6143], [2, 2, 0]]
[[802, 6151], [2, 2, 0]]
[[803, 6163], [2, 2, 0]]
[[804, 6173], [6, 2, 0]]
[[805, 6197], [4, 2, 0]]
[[806, 6199], [2, 2, 0]]
[[807, 6203], [2, 2, 0]]
[[808, 6211], [2, 2, 0]]
[[809, 6217], [2, 2, 0]]
[[810, 6221], [2, 2, 0]]
[[811, 6229], [2, 2, 0]]
[[812, 6247], [2, 2, 0]]
[[813, 6257], [2, 2, 0]]
[[814, 6263], [2, 2, 0]]
[[815, 6269], [2, 2, 0]]
[[816, 6271], [2, 2, 0]]
[[817, 6277], [2, 2, 0]]
[[818, 6287], [2, 2, 0]]
[[819, 6299], [2, 2, 0]]
[[820, 6301], [6, 2, 0]]
[[821, 6311], [2, 2, 0]]
[[822, 6317], [6, 2, 0]]
[[823, 6323], [2, 2, 0]]
[[824, 6329], [2, 2, 0]]
[[825, 6337], [2, 2, 0]]
[[826, 6343], [2, 2, 0]]
[[827, 6353], [2, 2, 0]]
[[828, 6359], [4, 2, 0]]
[[829, 6361], [2, 2, 0]]
[[830, 6367], [2, 2, 0]]
[[831, 6373], [2, 2, 0]]
[[832, 6379], [2, 2, 0]]
[[833, 6389], [4, 2, 0]]
[[834, 6397], [2, 2, 0]]
[[835, 6421], [2, 2, 0]]
[[836, 6427], [2, 2, 0]]
[[837, 6449], [2, 2, 0]]
[[838, 6451], [2, 2, 0]]
[[839, 6469], [2, 2, 0]]
[[840, 6473], [2, 2, 0]]
[[841, 6481], [4, 2, 0]]
[[842, 6491], [2, 2, 0]]
[[843, 6521], [2, 2, 0]]
[[844, 6529], [2, 2, 0]]
[[845, 6547], [2, 2, 0]]
[[846, 6551], [2, 2, 0]]
[[847, 6553], [2, 2, 0]]
[[848, 6563], "T0"]
[[849, 6569], [2, 2, 0]]
[[850, 6571], [2, 2, 0]]
[[851, 6577], [2, 2, 0]]
[[852, 6581], [2, 2, 0]]
[[853, 6599], [2, 2, 0]]
[[854, 6607], [2, 2, 0]]
[[855, 6619], [2, 2, 0]]
[[856, 6637], [2, 2, 0]]
[[857, 6653], [6, 2, 0]]
[[858, 6659], "?"]
[[859, 6661], [2, 2, 0]]
[[860, 6673], [2, 2, 0]]
[[861, 6679], [2, 2, 0]]
[[862, 6689], [2, 2, 0]]
[[863, 6691], [2, 2, 0]]
[[864, 6701], [2, 2, 0]]
[[865, 6703], [2, 2, 0]]
[[866, 6709], [2, 2, 0]]
[[867, 6719], [2, 2, 0]]
[[868, 6733], [4, 2, 0]]
[[869, 6737], [2, 2, 0]]
[[870, 6761], [2, 2, 0]]
[[871, 6763], [6, 2, 0]]
[[872, 6779], [4, 4, 0]]
[[873, 6781], [2, 2, 0]]
[[874, 6791], [2, 2, 0]]
[[875, 6793], [2, 2, 0]]
[[876, 6803], [2, 2, 0]]
[[877, 6823], [2, 2, 0]]
[[878, 6827], [4, 4, 0]]
[[879, 6829], [2, 2, 0]]
[[880, 6833], [2, 2, 0]]
[[881, 6841], [2, 2, 0]]
[[882, 6857], [2, 2, 0]]
[[883, 6863], [2, 2, 0]]
[[884, 6869], [2, 2, 0]]
[[885, 6871], [2, 2, 0]]
[[886, 6883], [2, 2, 0]]
[[887, 6899], "T0"]
[[888, 6907], [2, 2, 0]]
[[889, 6911], [2, 2, 0]]
[[890, 6917], [4, 2, 0]]
[[891, 6947], [2, 2, 0]]
[[892, 6949], [2, 2, 0]]
[[893, 6959], [6, 2, 0]]
[[894, 6961], [2, 2, 0]]
[[895, 6967], [6, 2, 0]]
[[896, 6971], [2, 2, 0]]
[[897, 6977], [2, 2, 0]]
[[898, 6983], [2, 2, 0]]
[[899, 6991], [2, 2, 0]]
[[900, 6997], [2, 2, 0]]
[[901, 7001], [4, 2, 0]]
[[902, 7013], [4, 2, 0]]
[[903, 7019], [2, 2, 0]]
[[904, 7027], [2, 2, 0]]
[[905, 7039], [2, 2, 0]]
[[906, 7043], [2, 2, 0]]
[[907, 7057], [4, 2, 0]]
[[908, 7069], [2, 2, 0]]
[[909, 7079], [2, 2, 0]]
[[910, 7103], [2, 2, 0]]
[[911, 7109], [3, 2, 0]]
[[912, 7121], [2, 2, 0]]
[[913, 7127], [2, 2, 0]]
[[914, 7129], [6, 2, 0]]
[[915, 7151], [4, 2, 0]]
[[916, 7159], [2, 2, 0]]
[[917, 7177], [2, 2, 0]]
[[918, 7187], "T0"]
[[919, 7193], [2, 2, 0]]
[[920, 7207], [2, 2, 0]]
[[921, 7211], [2, 2, 0]]
[[922, 7213], [2, 2, 0]]
[[923, 7219], [2, 2, 0]]
[[924, 7229], [2, 2, 0]]
[[925, 7237], [6, 2, 0]]
[[926, 7243], [2, 2, 0]]
[[927, 7247], [2, 2, 0]]
[[928, 7253], [2, 2, 0]]
[[929, 7283], [2, 2, 0]]
[[930, 7297], [4, 2, 0]]
[[931, 7307], [4, 4, 0]]
[[932, 7309], [2, 2, 0]]
[[933, 7321], [2, 2, 0]]
[[934, 7331], [2, 2, 0]]
[[935, 7333], [2, 2, 0]]
[[936, 7349], [2, 2, 0]]
[[937, 7351], [10, 2, 0]]
[[938, 7369], [2, 2, 0]]
[[939, 7393], [2, 2, 0]]
[[940, 7411], [2, 2, 0]]
[[941, 7417], [2, 2, 0]]
[[942, 7433], [2, 2, 0]]
[[943, 7451], [4, 2, 0]]
[[944, 7457], [2, 2, 0]]
[[945, 7459], [2, 2, 0]]
[[946, 7477], [2, 2, 0]]
[[947, 7481], [2, 2, 0]]
[[948, 7487], [2, 2, 0]]
[[949, 7489], [2, 2, 0]]
[[950, 7499], [2, 2, 0]]
[[951, 7507], [6, 2, 0]]
[[952, 7517], [3, 2, 0]]
[[953, 7523], [16, 16, 1]]
[[954, 7529], [2, 2, 0]]
[[955, 7537], [2, 2, 0]]
[[956, 7541], [2, 2, 0]]
[[957, 7547], [10, 2, 0]]
[[958, 7549], [2, 2, 0]]
[[959, 7559], [2, 2, 0]]
[[960, 7561], [4, 2, 0]]
[[961, 7573], [2, 2, 0]]
[[962, 7577], "S0"]
[[963, 7583], [2, 2, 0]]
[[964, 7589], [2, 2, 0]]
[[965, 7591], [2, 2, 0]]
[[966, 7603], [2, 2, 0]]
[[967, 7607], [2, 2, 0]]
[[968, 7621], [2, 2, 0]]
[[969, 7639], [2, 2, 0]]
[[970, 7643], [4, 4, 0]]
[[971, 7649], [2, 2, 0]]
[[972, 7669], [2, 2, 0]]
[[973, 7673], [2, 2, 0]]
[[974, 7681], [4, 2, 0]]
[[975, 7687], [2, 2, 0]]
[[976, 7691], [2, 2, 0]]
[[977, 7699], [2, 2, 0]]
[[978, 7703], [2, 2, 0]]
[[979, 7717], [2, 2, 0]]
[[980, 7723], [2, 2, 0]]
[[981, 7727], [2, 2, 0]]
[[982, 7741], [4, 2, 0]]
[[983, 7753], [2, 2, 0]]
[[984, 7757], [2, 2, 0]]
[[985, 7759], [2, 2, 0]]
[[986, 7789], [2, 2, 0]]
[[987, 7793], [2, 2, 0]]
[[988, 7817], [2, 2, 0]]
[[989, 7823], [2, 2, 0]]
[[990, 7829], [2, 2, 0]]
[[991, 7841], [2, 2, 0]]
[[992, 7853], [2, 2, 0]]
[[993, 7867], [2, 2, 0]]
[[994, 7873], [2, 2, 0]]
[[995, 7877], "T0"]
[[996, 7879], [2, 2, 0]]
[[997, 7883], [2, 2, 0]]
[[998, 7901], [2, 2, 0]]
[[999, 7907], [2, 2, 0]]
[[1000, 7919], [2, 2, 0]]
time = 2h, 3min, 9,684 ms.
\end{verbatim}
\end{multicols}
}

\newpage
{\tiny
\begin{multicols}{4}
\begin{verbatim}
gp > NPs(1001,1500)
[[1001, 7927], [2, 2, 0]]
[[1002, 7933], [2, 2, 0]]
[[1003, 7937], [6, 2, 0]]
[[1004, 7949], [2, 2, 0]]
[[1005, 7951], [2, 2, 0]]
[[1006, 7963], [2, 2, 0]]
[[1007, 7993], [2, 2, 0]]
[[1008, 8009], [2, 2, 0]]
[[1009, 8011], [2, 2, 0]]
[[1010, 8017], [2, 2, 0]]
[[1011, 8039], [2, 2, 0]]
[[1012, 8053], [2, 2, 0]]
[[1013, 8059], [2, 2, 0]]
[[1014, 8069], [4, 2, 0]]
[[1015, 8081], [2, 2, 0]]
[[1016, 8087], [2, 2, 0]]
[[1017, 8089], [2, 2, 0]]
[[1018, 8093], [2, 2, 0]]
[[1019, 8101], [6, 2, 0]]
[[1020, 8111], [2, 2, 0]]
[[1021, 8117], [2, 2, 0]]
[[1022, 8123], [2, 2, 0]]
[[1023, 8147], [4, 4, 0]]
[[1024, 8161], [2, 2, 0]]
[[1025, 8167], [2, 2, 0]]
[[1026, 8171], [2, 2, 0]]
[[1027, 8179], [2, 2, 0]]
[[1028, 8191], [4, 2, 0]]
[[1029, 8209], [2, 2, 0]]
[[1030, 8219], [2, 2, 0]]
[[1031, 8221], [2, 2, 0]]
[[1032, 8231], [2, 2, 0]]
[[1033, 8233], [6, 2, 0]]
[[1034, 8237], [2, 2, 0]]
[[1035, 8243], [4, 4, 0]]
[[1036, 8263], [4, 2, 0]]
[[1037, 8269], [2, 2, 0]]
[[1038, 8273], [2, 2, 0]]
[[1039, 8287], [2, 2, 0]]
[[1040, 8291], [2, 4, 0]]
[[1041, 8293], [2, 2, 0]]
[[1042, 8297], [2, 2, 0]]
[[1043, 8311], [2, 2, 0]]
[[1044, 8317], [2, 2, 0]]
[[1045, 8329], [2, 2, 0]]
[[1046, 8353], [2, 2, 0]]
[[1047, 8363], [4, 4, 0]]
[[1048, 8369], [2, 2, 0]]
[[1049, 8377], [2, 2, 0]]
[[1050, 8387], [2, 2, 0]]
[[1051, 8389], [2, 2, 0]]
[[1052, 8419], [2, 2, 0]]
[[1053, 8423], [2, 2, 0]]
[[1054, 8429], [6, 2, 0]]
[[1055, 8431], [2, 2, 0]]
[[1056, 8443], [2, 2, 0]]
[[1057, 8447], [2, 2, 0]]
[[1058, 8461], [2, 2, 0]]
[[1059, 8467], [2, 2, 0]]
[[1060, 8501], [2, 2, 0]]
[[1061, 8513], [2, 2, 0]]
[[1062, 8521], [2, 2, 0]]
[[1063, 8527], [2, 2, 0]]
[[1064, 8537], [2, 2, 0]]
[[1065, 8539], [2, 2, 0]]
[[1066, 8543], [2, 2, 0]]
[[1067, 8563], [2, 2, 0]]
[[1068, 8573], [2, 2, 0]]
[[1069, 8581], [2, 2, 0]]
[[1070, 8597], [2, 2, 0]]
[[1071, 8599], [2, 2, 0]]
[[1072, 8609], [2, 2, 0]]
[[1073, 8623], [2, 2, 0]]
[[1074, 8627], [2, 2, 0]]
[[1075, 8629], [2, 2, 0]]
[[1076, 8641], [4, 2, 0]]
[[1077, 8647], [2, 2, 0]]
[[1078, 8663], [2, 2, 0]]
[[1079, 8669], [2, 2, 0]]
[[1080, 8677], [2, 2, 0]]
[[1081, 8681], [2, 2, 0]]
[[1082, 8689], [2, 2, 0]]
[[1083, 8693], [2, 2, 0]]
[[1084, 8699], [4, 4, 0]]
[[1085, 8707], [2, 2, 0]]
[[1086, 8713], [4, 2, 0]]
[[1087, 8719], [2, 2, 0]]
[[1088, 8731], [2, 2, 0]]
[[1089, 8737], [2, 2, 0]]
[[1090, 8741], [2, 2, 0]]
[[1091, 8747], [4, 4, 0]]
[[1092, 8753], [2, 2, 0]]
[[1093, 8761], [2, 2, 0]]
[[1094, 8779], [2, 2, 0]]
[[1095, 8783], [2, 2, 0]]
[[1096, 8803], [3, 2, 0]]
[[1097, 8807], [2, 2, 0]]
[[1098, 8819], [2, 2, 0]]
[[1099, 8821], [4, 2, 0]]
[[1100, 8831], [2, 2, 0]]
[[1101, 8837], "T1"]
[[1102, 8839], [2, 2, 0]]
[[1103, 8849], [2, 2, 0]]
[[1104, 8861], [2, 2, 0]]
[[1105, 8863], [2, 2, 0]]
[[1106, 8867], [2, 2, 0]]
[[1107, 8887], [2, 2, 0]]
[[1108, 8893], [2, 2, 0]]
[[1109, 8923], [2, 2, 0]]
[[1110, 8929], [2, 2, 0]]
[[1111, 8933], [2, 2, 0]]
[[1112, 8941], [2, 2, 0]]
[[1113, 8951], [2, 2, 0]]
[[1114, 8963], [2, 2, 0]]
[[1115, 8969], [2, 2, 0]]
[[1116, 8971], [2, 2, 0]]
[[1117, 8999], [2, 2, 0]]
[[1118, 9001], [4, 2, 0]]
[[1119, 9007], [2, 2, 0]]
[[1120, 9011], [2, 2, 0]]
[[1121, 9013], [2, 2, 0]]
[[1122, 9029], [2, 2, 0]]
[[1123, 9041], [2, 2, 0]]
[[1124, 9043], [2, 2, 0]]
[[1125, 9049], [2, 2, 0]]
[[1126, 9059], [2, 2, 0]]
[[1127, 9067], [2, 2, 0]]
[[1128, 9091], [2, 2, 0]]
[[1129, 9103], [2, 2, 0]]
[[1130, 9109], [2, 2, 0]]
[[1131, 9127], [2, 2, 0]]
[[1132, 9133], [2, 2, 0]]
[[1133, 9137], [2, 2, 0]]
[[1134, 9151], [2, 2, 0]]
[[1135, 9157], [2, 2, 0]]
[[1136, 9161], [2, 2, 0]]
[[1137, 9173], [4, 2, 0]]
[[1138, 9181], [2, 2, 0]]
[[1139, 9187], [2, 2, 0]]
[[1140, 9199], [2, 2, 0]]
[[1141, 9203], [2, 2, 0]]
[[1142, 9209], [2, 2, 0]]
[[1143, 9221], [2, 2, 0]]
[[1144, 9227], [2, 2, 0]]
[[1145, 9239], [2, 2, 0]]
[[1146, 9241], [2, 2, 0]]
[[1147, 9257], [2, 2, 0]]
[[1148, 9277], [2, 2, 0]]
[[1149, 9281], [2, 2, 0]]
[[1150, 9283], [2, 2, 0]]
[[1151, 9293], [2, 2, 0]]
[[1152, 9311], [4, 2, 0]]
[[1153, 9319], [2, 2, 0]]
[[1154, 9323], [2, 2, 0]]
[[1155, 9337], [2, 2, 0]]
[[1156, 9341], [2, 2, 0]]
[[1157, 9343], [2, 2, 0]]
[[1158, 9349], [2, 2, 0]]
[[1159, 9371], [3, 3, 0]]
[[1160, 9377], [2, 2, 0]]
[[1161, 9391], [2, 2, 0]]
[[1162, 9397], [2, 2, 0]]
[[1163, 9403], [2, 2, 0]]
[[1164, 9413], [2, 2, 0]]
[[1165, 9419], [3, 3, 0]]
[[1166, 9421], [2, 2, 0]]
[[1167, 9431], [2, 2, 0]]
[[1168, 9433], [2, 2, 0]]
[[1169, 9437], [2, 2, 0]]
[[1170, 9439], [2, 2, 0]]
[[1171, 9461], [2, 2, 0]]
[[1172, 9463], [2, 2, 0]]
[[1173, 9467], [4, 4, 0]]
[[1174, 9473], [2, 2, 0]]
[[1175, 9479], [2, 2, 0]]
[[1176, 9491], [3, 3, 0]]
[[1177, 9497], "S0"]
[[1178, 9511], [2, 2, 0]]
[[1179, 9521], [2, 2, 0]]
[[1180, 9533], "S0"]
[[1181, 9539], [2, 2, 0]]
[[1182, 9547], [2, 2, 0]]
[[1183, 9551], [2, 2, 0]]
[[1184, 9587], [4, 4, 0]]
[[1185, 9601], [4, 2, 0]]
[[1186, 9613], [2, 2, 0]]
[[1187, 9619], [2, 2, 0]]
[[1188, 9623], [2, 2, 0]]
[[1189, 9629], [2, 2, 0]]
[[1190, 9631], [2, 2, 0]]
[[1191, 9643], [2, 2, 0]]
[[1192, 9649], [2, 2, 0]]
[[1193, 9661], [2, 2, 0]]
[[1194, 9677], [2, 2, 0]]
[[1195, 9679], [2, 2, 0]]
[[1196, 9689], [2, 2, 0]]
[[1197, 9697], [2, 2, 0]]
[[1198, 9719], [2, 2, 0]]
[[1199, 9721], [6, 2, 0]]
[[1200, 9733], [2, 2, 0]]
[[1201, 9739], [2, 2, 0]]
[[1202, 9743], [2, 2, 0]]
[[1203, 9749], [3, 2, 0]]
[[1204, 9767], [2, 2, 0]]
[[1205, 9769], [2, 2, 0]]
[[1206, 9781], [2, 2, 0]]
[[1207, 9787], [2, 2, 0]]
[[1208, 9791], [2, 2, 0]]
[[1209, 9803], [4, 2, 0]]
[[1210, 9811], [2, 2, 0]]
[[1211, 9817], [2, 2, 0]]
[[1212, 9829], [4, 2, 0]]
[[1213, 9833], [2, 2, 0]]
[[1214, 9839], [2, 2, 0]]
[[1215, 9851], [2, 2, 0]]
[[1216, 9857], [4, 2, 0]]
[[1217, 9859], [2, 2, 0]]
[[1218, 9871], [2, 2, 0]]
[[1219, 9883], [2, 2, 0]]
[[1220, 9887], [2, 2, 0]]
[[1221, 9901], [4, 2, 0]]
[[1222, 9907], [2, 2, 0]]
[[1223, 9923], [2, 2, 0]]
[[1224, 9929], [2, 2, 0]]
[[1225, 9931], [2, 2, 0]]
[[1226, 9941], [2, 2, 0]]
[[1227, 9949], [2, 2, 0]]
[[1228, 9967], [2, 2, 0]]
[[1229, 9973], [2, 2, 0]]
[[1230, 10007], [2, 2, 0]]
[[1231, 10009], [2, 2, 0]]
[[1232, 10037], [2, 2, 0]]
[[1233, 10039], [2, 2, 0]]
[[1234, 10061], [2, 2, 0]]
[[1235, 10067], [2, 2, 0]]
[[1236, 10069], [2, 2, 0]]
[[1237, 10079], [2, 2, 0]]
[[1238, 10091], [2, 2, 0]]
[[1239, 10093], [2, 2, 0]]
[[1240, 10099], [4, 2, 0]]
[[1241, 10103], [2, 2, 0]]
[[1242, 10111], [2, 2, 0]]
[[1243, 10133], [2, 2, 0]]
[[1244, 10139], [4, 2, 0]]
[[1245, 10141], [2, 2, 0]]
[[1246, 10151], [2, 2, 0]]
[[1247, 10159], [2, 2, 0]]
[[1248, 10163], [2, 2, 0]]
[[1249, 10169], [2, 2, 0]]
[[1250, 10177], [2, 2, 0]]
[[1251, 10181], [2, 2, 0]]
[[1252, 10193], [2, 2, 0]]
[[1253, 10211], [4, 2, 0]]
[[1254, 10223], [2, 2, 0]]
[[1255, 10243], [2, 2, 0]]
[[1256, 10247], [2, 2, 0]]
[[1257, 10253], [2, 2, 0]]
[[1258, 10259], [2, 2, 0]]
[[1259, 10267], [2, 2, 0]]
[[1260, 10271], [2, 2, 0]]
[[1261, 10273], [2, 2, 0]]
[[1262, 10289], [2, 2, 0]]
[[1263, 10301], [2, 2, 0]]
[[1264, 10303], [2, 2, 0]]
[[1265, 10313], [2, 2, 0]]
[[1266, 10321], [2, 2, 0]]
[[1267, 10331], [4, 4, 0]]
[[1268, 10333], [2, 2, 0]]
[[1269, 10337], [2, 2, 0]]
[[1270, 10343], [2, 2, 0]]
[[1271, 10357], [2, 2, 0]]
[[1272, 10369], [8, 2, 0]]
[[1273, 10391], [2, 2, 0]]
[[1274, 10399], [2, 2, 0]]
[[1275, 10427], [2, 2, 0]]
[[1276, 10429], [2, 2, 0]]
[[1277, 10433], [2, 2, 0]]
[[1278, 10453], [2, 2, 0]]
[[1279, 10457], "S0"]
[[1280, 10459], [2, 2, 0]]
[[1281, 10463], [2, 2, 0]]
[[1282, 10477], [4, 2, 0]]
[[1283, 10487], [6, 2, 0]]
[[1284, 10499], [2, 2, 0]]
[[1285, 10501], [4, 2, 0]]
[[1286, 10513], [2, 2, 0]]
[[1287, 10529], [2, 2, 0]]
[[1288, 10531], [2, 2, 0]]
[[1289, 10559], [2, 2, 0]]
[[1290, 10567], [2, 2, 0]]
[[1291, 10589], [6, 2, 0]]
[[1292, 10597], [2, 2, 0]]
[[1293, 10601], [2, 2, 0]]
[[1294, 10607], [2, 2, 0]]
[[1295, 10613], [2, 2, 0]]
[[1296, 10627], [2, 2, 0]]
[[1297, 10631], [2, 2, 0]]
[[1298, 10639], [2, 2, 0]]
[[1299, 10651], [2, 2, 0]]
[[1300, 10657], [2, 2, 0]]
[[1301, 10663], [2, 2, 0]]
[[1302, 10667], "?"]
[[1303, 10687], [2, 2, 0]]
[[1304, 10691], [4, 2, 0]]
[[1305, 10709], [2, 2, 0]]
[[1306, 10711], [2, 2, 0]]
[[1307, 10723], [2, 2, 0]]
[[1308, 10729], [2, 2, 0]]
[[1309, 10733], [6, 2, 0]]
[[1310, 10739], [2, 2, 0]]
[[1311, 10753], [2, 2, 0]]
[[1312, 10771], [2, 2, 0]]
[[1313, 10781], [2, 2, 0]]
[[1314, 10789], [2, 2, 0]]
[[1315, 10799], [2, 2, 0]]
[[1316, 10831], [6, 2, 0]]
[[1317, 10837], [2, 2, 0]]
[[1318, 10847], [2, 2, 0]]
[[1319, 10853], [2, 2, 0]]
[[1320, 10859], [4, 4, 0]]
[[1321, 10861], [2, 2, 0]]
[[1322, 10867], [2, 2, 0]]
[[1323, 10883], [10, 10, 1]]
[[1324, 10889], [2, 2, 0]]
[[1325, 10891], [2, 2, 0]]
[[1326, 10903], [2, 2, 0]]
[[1327, 10909], [4, 2, 0]]
[[1328, 10937], "S0"]
[[1329, 10939], [2, 2, 0]]
[[1330, 10949], [2, 2, 0]]
[[1331, 10957], [2, 2, 0]]
[[1332, 10973], [4, 2, 0]]
[[1333, 10979], [2, 2, 0]]
[[1334, 10987], [2, 2, 0]]
[[1335, 10993], [2, 2, 0]]
[[1336, 11003], [4, 4, 0]]
[[1337, 11027], [2, 4, 0]]
[[1338, 11047], [2, 2, 0]]
[[1339, 11057], [2, 2, 0]]
[[1340, 11059], [2, 2, 0]]
[[1341, 11069], [6, 2, 0]]
[[1342, 11071], [2, 2, 0]]
[[1343, 11083], [2, 2, 0]]
[[1344, 11087], [2, 2, 0]]
[[1345, 11093], [2, 2, 0]]
[[1346, 11113], [2, 2, 0]]
[[1347, 11117], [2, 2, 0]]
[[1348, 11119], [2, 2, 0]]
[[1349, 11131], [2, 2, 0]]
[[1350, 11149], [2, 2, 0]]
[[1351, 11159], [2, 2, 0]]
[[1352, 11161], [2, 2, 0]]
[[1353, 11171], [4, 4, 0]]
[[1354, 11173], [2, 2, 0]]
[[1355, 11177], [2, 2, 0]]
[[1356, 11197], [2, 2, 0]]
[[1357, 11213], [6, 2, 0]]
[[1358, 11239], [2, 2, 0]]
[[1359, 11243], [2, 2, 0]]
[[1360, 11251], [10, 2, 0]]
[[1361, 11257], [2, 2, 0]]
[[1362, 11261], [2, 2, 0]]
[[1363, 11273], [2, 2, 0]]
[[1364, 11279], [2, 2, 0]]
[[1365, 11287], [6, 2, 0]]
[[1366, 11299], [2, 2, 0]]
[[1367, 11311], [2, 2, 0]]
[[1368, 11317], [2, 2, 0]]
[[1369, 11321], [2, 2, 0]]
[[1370, 11329], [2, 2, 0]]
[[1371, 11351], [2, 2, 0]]
[[1372, 11353], [2, 2, 0]]
[[1373, 11369], [2, 2, 0]]
[[1374, 11383], [2, 2, 0]]
[[1375, 11393], [2, 2, 0]]
[[1376, 11399], [2, 2, 0]]
[[1377, 11411], [2, 2, 0]]
[[1378, 11423], [2, 2, 0]]
[[1379, 11437], [2, 2, 0]]
[[1380, 11443], "S0"]
[[1381, 11447], [2, 2, 0]]
[[1382, 11467], [6, 2, 0]]
[[1383, 11471], [2, 2, 0]]
[[1384, 11483], [2, 2, 0]]
[[1385, 11489], [2, 2, 0]]
[[1386, 11491], [2, 2, 0]]
[[1387, 11497], [2, 2, 0]]
[[1388, 11503], [2, 2, 0]]
[[1389, 11519], [2, 2, 0]]
[[1390, 11527], [2, 2, 0]]
[[1391, 11549], [6, 2, 0]]
[[1392, 11551], [2, 2, 0]]
[[1393, 11579], [2, 2, 0]]
[[1394, 11587], [2, 2, 0]]
[[1395, 11593], [2, 2, 0]]
[[1396, 11597], [2, 2, 0]]
[[1397, 11617], [2, 2, 0]]
[[1398, 11621], [2, 2, 0]]
[[1399, 11633], [2, 2, 0]]
[[1400, 11657], [2, 2, 0]]
[[1401, 11677], [2, 2, 0]]
[[1402, 11681], [2, 2, 0]]
[[1403, 11689], [2, 2, 0]]
[[1404, 11699], "S1"]
[[1405, 11701], [2, 2, 0]]
[[1406, 11717], [2, 2, 0]]
[[1407, 11719], [2, 2, 0]]
[[1408, 11731], [2, 2, 0]]
[[1409, 11743], [2, 2, 0]]
[[1410, 11777], [2, 2, 0]]
[[1411, 11779], [2, 2, 0]]
[[1412, 11783], [2, 2, 0]]
[[1413, 11789], [2, 2, 0]]
[[1414, 11801], [2, 2, 0]]
[[1415, 11807], [2, 2, 0]]
[[1416, 11813], [4, 2, 0]]
[[1417, 11821], [2, 2, 0]]
[[1418, 11827], [6, 2, 0]]
[[1419, 11831], [2, 2, 0]]
[[1420, 11833], [2, 2, 0]]
[[1421, 11839], [2, 2, 0]]
[[1422, 11863], [2, 2, 0]]
[[1423, 11867], [4, 4, 0]]
[[1424, 11887], [2, 2, 0]]
[[1425, 11897], "S0"]
[[1426, 11903], [2, 2, 0]]
[[1427, 11909], [2, 2, 0]]
[[1428, 11923], "S0"]
[[1429, 11927], [2, 2, 0]]
[[1430, 11933], [2, 2, 0]]
[[1431, 11939], [2, 2, 0]]
[[1432, 11941], [2, 2, 0]]
[[1433, 11953], [2, 2, 0]]
[[1434, 11959], [2, 2, 0]]
[[1435, 11969], [2, 2, 0]]
[[1436, 11971], [2, 2, 0]]
[[1437, 11981], [2, 2, 0]]
[[1438, 11987], [4, 4, 0]]
[[1439, 12007], [2, 2, 0]]
[[1440, 12011], [2, 4, 0]]
[[1441, 12037], [2, 2, 0]]
[[1442, 12041], [2, 2, 0]]
[[1443, 12043], [2, 2, 0]]
[[1444, 12049], [2, 2, 0]]
[[1445, 12071], [2, 2, 0]]
[[1446, 12073], [2, 2, 0]]
[[1447, 12097], [2, 2, 0]]
[[1448, 12101], [8, 2, 0]]
[[1449, 12107], [2, 2, 0]]
[[1450, 12109], [2, 2, 0]]
[[1451, 12113], [2, 2, 0]]
[[1452, 12119], [2, 2, 0]]
[[1453, 12143], [2, 2, 0]]
[[1454, 12149], [4, 2, 0]]
[[1455, 12157], [2, 2, 0]]
[[1456, 12161], [2, 2, 0]]
[[1457, 12163], [2, 2, 0]]
[[1458, 12197], "S0"]
[[1459, 12203], [4, 4, 0]]
[[1460, 12211], [2, 2, 0]]
[[1461, 12227], "?"]
[[1462, 12239], [2, 2, 0]]
[[1463, 12241], [2, 2, 0]]
[[1464, 12251], [10, 2, 0]]
[[1465, 12253], [2, 2, 0]]
[[1466, 12263], [2, 2, 0]]
[[1467, 12269], "S0"]
[[1468, 12277], [2, 2, 0]]
[[1469, 12281], [2, 2, 0]]
[[1470, 12289], [16, 2, 0]]
[[1471, 12301], [2, 2, 0]]
[[1472, 12323], [4, 4, 0]]
[[1473, 12329], [2, 2, 0]]
[[1474, 12343], [4, 2, 0]]
[[1475, 12347], [4, 4, 0]]
[[1476, 12373], [2, 2, 0]]
[[1477, 12377], [2, 2, 0]]
[[1478, 12379], [2, 2, 0]]
[[1479, 12391], [2, 2, 0]]
[[1480, 12401], [2, 2, 0]]
[[1481, 12409], [2, 2, 0]]
[[1482, 12413], [2, 2, 0]]
[[1483, 12421], [2, 2, 0]]
[[1484, 12433], [2, 2, 0]]
[[1485, 12437], [4, 2, 0]]
[[1486, 12451], [2, 2, 0]]
[[1487, 12457], [2, 2, 0]]
[[1488, 12473], [2, 2, 0]]
[[1489, 12479], [2, 2, 0]]
[[1490, 12487], [2, 2, 0]]
[[1491, 12491], [4, 4, 0]]
[[1492, 12497], [2, 2, 0]]
[[1493, 12503], [2, 2, 0]]
[[1494, 12511], [2, 2, 0]]
[[1495, 12517], [2, 2, 0]]
[[1496, 12527], [2, 2, 0]]
[[1497, 12539], [4, 4, 0]]
[[1498, 12541], [2, 2, 0]]
[[1499, 12547], [2, 2, 0]]
[[1500, 12553], [2, 2, 0]]
time = 7h, 58min, 27,811 ms.
\end{verbatim}
\end{multicols}
}

\newpage
{\tiny
\begin{multicols}{4}
\begin{verbatim}
gp > NPs(1501,2000)
[[1501, 12569], [2, 2, 0]]
[[1502, 12577], [2, 2, 0]]
[[1503, 12583], [2, 2, 0]]
[[1504, 12589], [2, 2, 0]]
[[1505, 12601], [4, 2, 0]]
[[1506, 12611], [3, 3, 0]]
[[1507, 12613], [2, 2, 0]]
[[1508, 12619], [2, 2, 0]]
[[1509, 12637], [6, 2, 0]]
[[1510, 12641], [2, 2, 0]]
[[1511, 12647], [2, 2, 0]]
[[1512, 12653], [2, 2, 0]]
[[1513, 12659], "S1"]
[[1514, 12671], [2, 2, 0]]
[[1515, 12689], [2, 2, 0]]
[[1516, 12697], [2, 2, 0]]
[[1517, 12703], [2, 2, 0]]
[[1518, 12713], [2, 2, 0]]
[[1519, 12721], [2, 2, 0]]
[[1520, 12739], [2, 2, 0]]
[[1521, 12743], [2, 2, 0]]
[[1522, 12757], [2, 2, 0]]
[[1523, 12763], [2, 2, 0]]
[[1524, 12781], [2, 2, 0]]
[[1525, 12791], [2, 2, 0]]
[[1526, 12799], [3, 2, 0]]
[[1527, 12809], [2, 2, 0]]
[[1528, 12821], [2, 2, 0]]
[[1529, 12823], [2, 2, 0]]
[[1530, 12829], [2, 2, 0]]
[[1531, 12841], [2, 2, 0]]
[[1532, 12853], [2, 2, 0]]
[[1533, 12889], [2, 2, 0]]
[[1534, 12893], [2, 2, 0]]
[[1535, 12899], "S1"]
[[1536, 12907], [2, 2, 0]]
[[1537, 12911], [2, 2, 0]]
[[1538, 12917], [4, 2, 0]]
[[1539, 12919], [2, 2, 0]]
[[1540, 12923], [2, 2, 0]]
[[1541, 12941], [2, 2, 0]]
[[1542, 12953], [2, 2, 0]]
[[1543, 12959], [2, 2, 0]]
[[1544, 12967], [2, 2, 0]]
[[1545, 12973], [2, 2, 0]]
[[1546, 12979], [3, 2, 0]]
[[1547, 12983], [2, 2, 0]]
[[1548, 13001], [2, 2, 0]]
[[1549, 13003], [2, 2, 0]]
[[1550, 13007], [2, 2, 0]]
[[1551, 13009], [2, 2, 0]]
[[1552, 13033], [2, 2, 0]]
[[1553, 13037], "S0"]
[[1554, 13043], "S1"]
[[1555, 13049], [2, 2, 0]]
[[1556, 13063], [2, 2, 0]]
[[1557, 13093], [2, 2, 0]]
[[1558, 13099], [2, 2, 0]]
[[1559, 13103], [2, 2, 0]]
[[1560, 13109], [2, 2, 0]]
[[1561, 13121], [2, 2, 0]]
[[1562, 13127], [2, 2, 0]]
[[1563, 13147], [2, 2, 0]]
[[1564, 13151], [2, 2, 0]]
[[1565, 13159], [2, 2, 0]]
[[1566, 13163], [4, 4, 0]]
[[1567, 13171], [2, 2, 0]]
[[1568, 13177], [2, 2, 0]]
[[1569, 13183], [26, 2, 1]]
[[1570, 13187], [2, 2, 0]]
[[1571, 13217], [2, 2, 0]]
[[1572, 13219], "S0"]
[[1573, 13229], [6, 2, 0]]
[[1574, 13241], [2, 2, 0]]
[[1575, 13249], [4, 2, 0]]
[[1576, 13259], [2, 2, 0]]
[[1577, 13267], [2, 2, 0]]
[[1578, 13291], [2, 2, 0]]
[[1579, 13297], [2, 2, 0]]
[[1580, 13309], [2, 2, 0]]
[[1581, 13313], [4, 2, 0]]
[[1582, 13327], [2, 2, 0]]
[[1583, 13331], [2, 2, 0]]
[[1584, 13337], "S0"]
[[1585, 13339], [2, 2, 0]]
[[1586, 13367], [2, 2, 0]]
[[1587, 13381], [2, 2, 0]]
[[1588, 13397], [2, 2, 0]]
[[1589, 13399], [2, 2, 0]]
[[1590, 13411], [2, 2, 0]]
[[1591, 13417], [2, 2, 0]]
[[1592, 13421], [2, 2, 0]]
[[1593, 13441], [4, 2, 0]]
[[1594, 13451], [2, 4, 0]]
[[1595, 13457], [2, 2, 0]]
[[1596, 13463], [2, 2, 0]]
[[1597, 13469], [2, 2, 0]]
[[1598, 13477], [2, 2, 0]]
[[1599, 13487], [2, 2, 0]]
[[1600, 13499], [3, 3, 0]]
[[1601, 13513], [2, 2, 0]]
[[1602, 13523], [8, 8, 1]]
[[1603, 13537], [2, 2, 0]]
[[1604, 13553], [2, 2, 0]]
[[1605, 13567], [2, 2, 0]]
[[1606, 13577], [2, 2, 0]]
[[1607, 13591], [2, 2, 0]]
[[1608, 13597], [2, 2, 0]]
[[1609, 13613], [2, 2, 0]]
[[1610, 13619], [2, 2, 0]]
[[1611, 13627], [2, 2, 0]]
[[1612, 13633], [2, 2, 0]]
[[1613, 13649], [2, 2, 0]]
[[1614, 13669], [2, 2, 0]]
[[1615, 13679], [2, 2, 0]]
[[1616, 13681], [2, 2, 0]]
[[1617, 13687], [2, 2, 0]]
[[1618, 13691], [4, 4, 0]]
[[1619, 13693], [2, 2, 0]]
[[1620, 13697], [2, 2, 0]]
[[1621, 13709], [2, 2, 0]]
[[1622, 13711], [2, 2, 0]]
[[1623, 13721], [4, 2, 0]]
[[1624, 13723], [2, 2, 0]]
[[1625, 13729], [2, 2, 0]]
[[1626, 13751], [2, 2, 0]]
[[1627, 13757], [2, 2, 0]]
[[1628, 13759], [2, 2, 0]]
[[1629, 13763], [2, 2, 0]]
[[1630, 13781], [2, 2, 0]]
[[1631, 13789], [2, 2, 0]]
[[1632, 13799], [2, 2, 0]]
[[1633, 13807], [2, 2, 0]]
[[1634, 13829], [4, 2, 0]]
[[1635, 13831], [2, 2, 0]]
[[1636, 13841], [2, 2, 0]]
[[1637, 13859], [4, 4, 0]]
[[1638, 13873], [2, 2, 0]]
[[1639, 13877], [4, 2, 0]]
[[1640, 13879], [2, 2, 0]]
[[1641, 13883], [2, 2, 0]]
[[1642, 13901], [2, 2, 0]]
[[1643, 13903], [2, 2, 0]]
[[1644, 13907], [2, 2, 0]]
[[1645, 13913], [2, 2, 0]]
[[1646, 13921], [2, 2, 0]]
[[1647, 13931], [2, 2, 0]]
[[1648, 13933], [2, 2, 0]]
[[1649, 13963], [2, 2, 0]]
[[1650, 13967], [2, 2, 0]]
[[1651, 13997], "S0"]
[[1652, 13999], [2, 2, 0]]
[[1653, 14009], [2, 2, 0]]
[[1654, 14011], [2, 2, 0]]
[[1655, 14029], [2, 2, 0]]
[[1656, 14033], [2, 2, 0]]
[[1657, 14051], [4, 2, 0]]
[[1658, 14057], [2, 2, 0]]
[[1659, 14071], [2, 2, 0]]
[[1660, 14081], [4, 2, 0]]
[[1661, 14083], "S0"]
[[1662, 14087], [2, 2, 0]]
[[1663, 14107], [2, 2, 0]]
[[1664, 14143], [2, 2, 0]]
[[1665, 14149], [2, 2, 0]]
[[1666, 14153], [2, 2, 0]]
[[1667, 14159], [2, 2, 0]]
[[1668, 14173], [2, 2, 0]]
[[1669, 14177], [2, 2, 0]]
[[1670, 14197], [2, 2, 0]]
[[1671, 14207], [2, 2, 0]]
[[1672, 14221], [2, 2, 0]]
[[1673, 14243], "S1"]
[[1674, 14249], [2, 2, 0]]
[[1675, 14251], [2, 2, 0]]
[[1676, 14281], "?"]
[[1677, 14293], [2, 2, 0]]
[[1678, 14303], [2, 2, 0]]
[[1679, 14321], [2, 2, 0]]
[[1680, 14323], [2, 2, 0]]
[[1681, 14327], [2, 2, 0]]
[[1682, 14341], [2, 2, 0]]
[[1683, 14347], [2, 2, 0]]
[[1684, 14369], [2, 2, 0]]
[[1685, 14387], [8, 8, 1]]
[[1686, 14389], [2, 2, 0]]
[[1687, 14401], [4, 2, 0]]
[[1688, 14407], [14, 2, 0]]
[[1689, 14411], [2, 2, 0]]
[[1690, 14419], [6, 2, 0]]
[[1691, 14423], [2, 2, 0]]
[[1692, 14431], [2, 2, 0]]
[[1693, 14437], [2, 2, 0]]
[[1694, 14447], [2, 2, 0]]
[[1695, 14449], [2, 2, 0]]
[[1696, 14461], [2, 2, 0]]
[[1697, 14479], [2, 2, 0]]
[[1698, 14489], [2, 2, 0]]
[[1699, 14503], [2, 2, 0]]
[[1700, 14519], [2, 2, 0]]
[[1701, 14533], [2, 2, 0]]
[[1702, 14537], [2, 2, 0]]
[[1703, 14543], [2, 2, 0]]
[[1704, 14549], [4, 2, 0]]
[[1705, 14551], [2, 2, 0]]
[[1706, 14557], [2, 2, 0]]
[[1707, 14561], "T0"]
[[1708, 14563], [2, 2, 0]]
[[1709, 14591], [2, 2, 0]]
[[1710, 14593], [4, 2, 0]]
[[1711, 14621], [2, 2, 0]]
[[1712, 14627], [2, 2, 0]]
[[1713, 14629], [2, 2, 0]]
[[1714, 14633], [2, 2, 0]]
[[1715, 14639], [2, 2, 0]]
[[1716, 14653], [2, 2, 0]]
[[1717, 14657], [2, 2, 0]]
[[1718, 14669], [2, 2, 0]]
[[1719, 14683], [2, 2, 0]]
[[1720, 14699], [4, 4, 0]]
[[1721, 14713], [2, 2, 0]]
[[1722, 14717], [2, 2, 0]]
[[1723, 14723], "S1"]
[[1724, 14731], [2, 2, 0]]
[[1725, 14737], [2, 2, 0]]
[[1726, 14741], [2, 2, 0]]
[[1727, 14747], [4, 4, 0]]
[[1728, 14753], [2, 2, 0]]
[[1729, 14759], [2, 2, 0]]
[[1730, 14767], [2, 2, 0]]
[[1731, 14771], [2, 2, 0]]
[[1732, 14779], [2, 2, 0]]
[[1733, 14783], [2, 2, 0]]
[[1734, 14797], [2, 2, 0]]
[[1735, 14813], [2, 2, 0]]
[[1736, 14821], [2, 2, 0]]
[[1737, 14827], [2, 2, 0]]
[[1738, 14831], [2, 2, 0]]
[[1739, 14843], [4, 4, 0]]
[[1740, 14851], [2, 2, 0]]
[[1741, 14867], [8, 8, 1]]
[[1742, 14869], [2, 2, 0]]
[[1743, 14879], [2, 2, 0]]
[[1744, 14887], [2, 2, 0]]
[[1745, 14891], [3, 2, 0]]
[[1746, 14897], [2, 2, 0]]
[[1747, 14923], [2, 2, 0]]
[[1748, 14929], [2, 2, 0]]
[[1749, 14939], [2, 2, 0]]
[[1750, 14947], [2, 2, 0]]
[[1751, 14951], [2, 2, 0]]
[[1752, 14957], [2, 2, 0]]
[[1753, 14969], [2, 2, 0]]
[[1754, 14983], [2, 2, 0]]
[[1755, 15013], [6, 2, 0]]
[[1756, 15017], [2, 2, 0]]
[[1757, 15031], [2, 2, 0]]
[[1758, 15053], [2, 2, 0]]
[[1759, 15061], [2, 2, 0]]
[[1760, 15073], [2, 2, 0]]
[[1761, 15077], "S0"]
[[1762, 15083], [4, 4, 0]]
[[1763, 15091], [2, 2, 0]]
[[1764, 15101], [4, 2, 0]]
[[1765, 15107], [2, 2, 0]]
[[1766, 15121], [2, 2, 0]]
[[1767, 15131], [2, 4, 0]]
[[1768, 15137], [2, 2, 0]]
[[1769, 15139], [2, 2, 0]]
[[1770, 15149], [2, 2, 0]]
[[1771, 15161], [2, 2, 0]]
[[1772, 15173], [4, 2, 0]]
[[1773, 15187], [2, 2, 0]]
[[1774, 15193], [2, 2, 0]]
[[1775, 15199], [2, 2, 0]]
[[1776, 15217], [2, 2, 0]]
[[1777, 15227], [2, 2, 0]]
[[1778, 15233], [2, 2, 0]]
[[1779, 15241], [2, 2, 0]]
[[1780, 15259], [2, 2, 0]]
[[1781, 15263], [2, 2, 0]]
[[1782, 15269], [2, 2, 0]]
[[1783, 15271], [2, 2, 0]]
[[1784, 15277], [2, 2, 0]]
[[1785, 15287], [2, 2, 0]]
[[1786, 15289], [2, 2, 0]]
[[1787, 15299], "?"]
[[1788, 15307], [2, 2, 0]]
[[1789, 15313], [2, 2, 0]]
[[1790, 15319], [2, 2, 0]]
[[1791, 15329], [2, 2, 0]]
[[1792, 15331], [2, 2, 0]]
[[1793, 15349], [2, 2, 0]]
[[1794, 15359], [2, 2, 0]]
[[1795, 15361], [4, 2, 0]]
[[1796, 15373], [2, 2, 0]]
[[1797, 15377], [2, 2, 0]]
[[1798, 15383], [2, 2, 0]]
[[1799, 15391], [2, 2, 0]]
[[1800, 15401], [2, 2, 0]]
[[1801, 15413], [4, 2, 0]]
[[1802, 15427], [2, 2, 0]]
[[1803, 15439], [2, 2, 0]]
[[1804, 15443], [2, 2, 0]]
[[1805, 15451], [4, 2, 0]]
[[1806, 15461], [2, 2, 0]]
[[1807, 15467], [2, 2, 0]]
[[1808, 15473], [2, 2, 0]]
[[1809, 15493], [2, 2, 0]]
[[1810, 15497], [2, 2, 0]]
[[1811, 15511], [2, 2, 0]]
[[1812, 15527], [2, 2, 0]]
[[1813, 15541], [2, 2, 0]]
[[1814, 15551], [2, 2, 0]]
[[1815, 15559], [2, 2, 0]]
[[1816, 15569], [2, 2, 0]]
[[1817, 15581], [2, 2, 0]]
[[1818, 15583], [2, 2, 0]]
[[1819, 15601], [2, 2, 0]]
[[1820, 15607], [6, 2, 0]]
[[1821, 15619], [2, 2, 0]]
[[1822, 15629], [6, 2, 0]]
[[1823, 15641], [2, 2, 0]]
[[1824, 15643], [2, 2, 0]]
[[1825, 15647], [2, 2, 0]]
[[1826, 15649], [2, 2, 0]]
[[1827, 15661], [2, 2, 0]]
[[1828, 15667], [2, 2, 0]]
[[1829, 15671], [2, 2, 0]]
[[1830, 15679], [2, 2, 0]]
[[1831, 15683], "S0"]
[[1832, 15727], [2, 2, 0]]
[[1833, 15731], [2, 2, 0]]
[[1834, 15733], [2, 2, 0]]
[[1835, 15737], [2, 2, 0]]
[[1836, 15739], [2, 2, 0]]
[[1837, 15749], [2, 2, 0]]
[[1838, 15761], [2, 2, 0]]
[[1839, 15767], [2, 2, 0]]
[[1840, 15773], "S0"]
[[1841, 15787], [2, 2, 0]]
[[1842, 15791], [2, 2, 0]]
[[1843, 15797], [2, 2, 0]]
[[1844, 15803], [4, 4, 0]]
[[1845, 15809], [2, 2, 0]]
[[1846, 15817], [2, 2, 0]]
[[1847, 15823], [2, 2, 0]]
[[1848, 15859], [2, 2, 0]]
[[1849, 15877], [6, 2, 0]]
[[1850, 15881], [2, 2, 0]]
[[1851, 15887], [2, 2, 0]]
[[1852, 15889], [2, 2, 0]]
[[1853, 15901], [2, 2, 0]]
[[1854, 15907], [2, 2, 0]]
[[1855, 15913], [2, 2, 0]]
[[1856, 15919], [2, 2, 0]]
[[1857, 15923], [2, 2, 0]]
[[1858, 15937], [2, 2, 0]]
[[1859, 15959], [2, 2, 0]]
[[1860, 15971], [4, 4, 0]]
[[1861, 15973], [10, 2, 0]]
[[1862, 15991], [2, 2, 0]]
[[1863, 16001], [4, 2, 0]]
[[1864, 16007], [2, 2, 0]]
[[1865, 16033], [2, 2, 0]]
[[1866, 16057], [2, 2, 0]]
[[1867, 16061], [2, 2, 0]]
[[1868, 16063], [2, 2, 0]]
[[1869, 16067], [3, 2, 0]]
[[1870, 16069], [2, 2, 0]]
[[1871, 16073], [2, 2, 0]]
[[1872, 16087], [2, 2, 0]]
[[1873, 16091], [2, 4, 0]]
[[1874, 16097], [2, 2, 0]]
[[1875, 16103], [2, 2, 0]]
[[1876, 16111], [2, 2, 0]]
[[1877, 16127], [2, 2, 0]]
[[1878, 16139], [2, 2, 0]]
[[1879, 16141], [2, 2, 0]]
[[1880, 16183], [2, 2, 0]]
[[1881, 16187], [4, 4, 0]]
[[1882, 16189], [2, 2, 0]]
[[1883, 16193], [2, 2, 0]]
[[1884, 16217], "S0"]
[[1885, 16223], [2, 2, 0]]
[[1886, 16229], "S0"]
[[1887, 16231], [2, 2, 0]]
[[1888, 16249], [2, 2, 0]]
[[1889, 16253], [2, 2, 0]]
[[1890, 16267], [2, 2, 0]]
[[1891, 16273], [2, 2, 0]]
[[1892, 16301], [2, 2, 0]]
[[1893, 16319], [2, 2, 0]]
[[1894, 16333], [2, 2, 0]]
[[1895, 16339], [2, 2, 0]]
[[1896, 16349], [2, 2, 0]]
[[1897, 16361], [2, 2, 0]]
[[1898, 16363], [2, 2, 0]]
[[1899, 16369], [2, 2, 0]]
[[1900, 16381], [2, 2, 0]]
[[1901, 16411], [2, 2, 0]]
[[1902, 16417], [2, 2, 0]]
[[1903, 16421], [2, 2, 0]]
[[1904, 16427], [2, 2, 0]]
[[1905, 16433], [2, 2, 0]]
[[1906, 16447], [2, 2, 0]]
[[1907, 16451], [2, 2, 0]]
[[1908, 16453], [2, 2, 0]]
[[1909, 16477], [2, 2, 0]]
[[1910, 16481], [2, 2, 0]]
[[1911, 16487], [2, 2, 0]]
[[1912, 16493], [2, 2, 0]]
[[1913, 16519], [2, 2, 0]]
[[1914, 16529], [2, 2, 0]]
[[1915, 16547], [16, 16, 1]]
[[1916, 16553], [2, 2, 0]]
[[1917, 16561], [2, 2, 0]]
[[1918, 16567], [2, 2, 0]]
[[1919, 16573], [2, 2, 0]]
[[1920, 16603], [2, 2, 0]]
[[1921, 16607], [2, 2, 0]]
[[1922, 16619], [2, 2, 0]]
[[1923, 16631], [2, 2, 0]]
[[1924, 16633], [2, 2, 0]]
[[1925, 16649], [2, 2, 0]]
[[1926, 16651], [2, 2, 0]]
[[1927, 16657], [2, 2, 0]]
[[1928, 16661], [4, 2, 0]]
[[1929, 16673], [2, 2, 0]]
[[1930, 16691], [2, 2, 0]]
[[1931, 16693], [2, 2, 0]]
[[1932, 16699], [2, 2, 0]]
[[1933, 16703], [2, 2, 0]]
[[1934, 16729], [2, 2, 0]]
[[1935, 16741], [2, 2, 0]]
[[1936, 16747], [2, 2, 0]]
[[1937, 16759], [2, 2, 0]]
[[1938, 16763], [4, 4, 0]]
[[1939, 16787], [2, 2, 0]]
[[1940, 16811], [4, 2, 0]]
[[1941, 16823], [2, 2, 0]]
[[1942, 16829], [2, 2, 0]]
[[1943, 16831], [2, 2, 0]]
[[1944, 16843], [2, 2, 0]]
[[1945, 16871], [2, 2, 0]]
[[1946, 16879], [2, 2, 0]]
[[1947, 16883], [2, 2, 0]]
[[1948, 16889], "S0"]
[[1949, 16901], [6, 2, 0]]
[[1950, 16903], [2, 2, 0]]
[[1951, 16921], [2, 2, 0]]
[[1952, 16927], [2, 2, 0]]
[[1953, 16931], [2, 2, 0]]
[[1954, 16937], [2, 2, 0]]
[[1955, 16943], [2, 2, 0]]
[[1956, 16963], [2, 2, 0]]
[[1957, 16979], [2, 4, 0]]
[[1958, 16981], [2, 2, 0]]
[[1959, 16987], [2, 2, 0]]
[[1960, 16993], [2, 2, 0]]
[[1961, 17011], [6, 2, 0]]
[[1962, 17021], [2, 2, 0]]
[[1963, 17027], "?"]
[[1964, 17029], [2, 2, 0]]
[[1965, 17033], [2, 2, 0]]
[[1966, 17041], [2, 2, 0]]
[[1967, 17047], [2, 2, 0]]
[[1968, 17053], [2, 2, 0]]
[[1969, 17077], [2, 2, 0]]
[[1970, 17093], [4, 2, 0]]
[[1971, 17099], [2, 2, 0]]
[[1972, 17107], [2, 2, 0]]
[[1973, 17117], [2, 2, 0]]
[[1974, 17123], "S0"]
[[1975, 17137], [2, 2, 0]]
[[1976, 17159], [2, 2, 0]]
[[1977, 17167], [2, 2, 0]]
[[1978, 17183], [10, 2, 0]]
[[1979, 17189], [4, 2, 0]]
[[1980, 17191], [2, 2, 0]]
[[1981, 17203], [2, 2, 0]]
[[1982, 17207], [2, 2, 0]]
[[1983, 17209], [2, 2, 0]]
[[1984, 17231], [2, 2, 0]]
[[1985, 17239], [2, 2, 0]]
[[1986, 17257], [2, 2, 0]]
[[1987, 17291], [2, 2, 0]]
[[1988, 17293], [2, 2, 0]]
[[1989, 17299], [6, 2, 0]]
[[1990, 17317], [2, 2, 0]]
[[1991, 17321], [2, 2, 0]]
[[1992, 17327], [2, 2, 0]]
[[1993, 17333], [2, 2, 0]]
[[1994, 17341], [2, 2, 0]]
[[1995, 17351], [2, 2, 0]]
[[1996, 17359], [2, 2, 0]]
[[1997, 17377], [2, 2, 0]]
[[1998, 17383], [2, 2, 0]]
[[1999, 17387], [4, 4, 0]]
[[2000, 17389], [2, 2, 0]]
time = 18h, 3min, 310 ms.
\end{verbatim}
\end{multicols}
}

\newpage
{\tiny
\begin{multicols}{4}
\begin{verbatim}
gp > NPs(2001,2262)
[[2001, 17393], [2, 2, 0]]
[[2002, 17401], [2, 2, 0]]
[[2003, 17417], [2, 2, 0]]
[[2004, 17419], [2, 2, 0]]
[[2005, 17431], [2, 2, 0]]
[[2006, 17443], [2, 2, 0]]
[[2007, 17449], [2, 2, 0]]
[[2008, 17467], [2, 2, 0]]
[[2009, 17471], [2, 2, 0]]
[[2010, 17477], [2, 2, 0]]
[[2011, 17483], [4, 4, 0]]
[[2012, 17489], [2, 2, 0]]
[[2013, 17491], [2, 2, 0]]
[[2014, 17497], [6, 2, 0]]
[[2015, 17509], [2, 2, 0]]
[[2016, 17519], [2, 2, 0]]
[[2017, 17539], [2, 2, 0]]
[[2018, 17551], [2, 2, 0]]
[[2019, 17569], [2, 2, 0]]
[[2020, 17573], "S0"]
[[2021, 17579], [2, 2, 0]]
[[2022, 17581], [2, 2, 0]]
[[2023, 17597], [2, 2, 0]]
[[2024, 17599], [2, 2, 0]]
[[2025, 17609], [2, 2, 0]]
[[2026, 17623], [2, 2, 0]]
[[2027, 17627], [2, 2, 0]]
[[2028, 17657], "S0"]
[[2029, 17659], [2, 2, 0]]
[[2030, 17669], "S0"]
[[2031, 17681], "?"]
[[2032, 17683], [2, 2, 0]]
[[2033, 17707], [2, 2, 0]]
[[2034, 17713], [2, 2, 0]]
[[2035, 17729], [2, 2, 0]]
[[2036, 17737], [2, 2, 0]]
[[2037, 17747], [2, 2, 0]]
[[2038, 17749], [2, 2, 0]]
[[2039, 17761], [2, 2, 0]]
[[2040, 17783], [2, 2, 0]]
[[2041, 17789], "S0"]
[[2042, 17791], [2, 2, 0]]
[[2043, 17807], [2, 2, 0]]
[[2044, 17827], "S0"]
[[2045, 17837], [4, 2, 0]]
[[2046, 17839], [2, 2, 0]]
[[2047, 17851], [2, 2, 0]]
[[2048, 17863], [2, 2, 0]]
[[2049, 17881], [2, 2, 0]]
[[2050, 17891], [2, 2, 0]]
[[2051, 17903], [2, 2, 0]]
[[2052, 17909], [2, 2, 0]]
[[2053, 17911], [2, 2, 0]]
[[2054, 17921], [4, 2, 0]]
[[2055, 17923], [2, 2, 0]]
[[2056, 17929], [2, 2, 0]]
[[2057, 17939], "S1"]
[[2058, 17957], [6, 2, 0]]
[[2059, 17959], [2, 2, 0]]
[[2060, 17971], [2, 2, 0]]
[[2061, 17977], [2, 2, 0]]
[[2062, 17981], [2, 2, 0]]
[[2063, 17987], [2, 2, 0]]
[[2064, 17989], [2, 2, 0]]
[[2065, 18013], [2, 2, 0]]
[[2066, 18041], "T0"]
[[2067, 18043], [2, 2, 0]]
[[2068, 18047], [2, 2, 0]]
[[2069, 18049], [2, 2, 0]]
[[2070, 18059], "?"]
[[2071, 18061], [2, 2, 0]]
[[2072, 18077], "S0"]
[[2073, 18089], [2, 2, 0]]
[[2074, 18097], "T0"]
[[2075, 18119], [2, 2, 0]]
[[2076, 18121], [2, 2, 0]]
[[2077, 18127], [2, 2, 0]]
[[2078, 18131], [2, 2, 0]]
[[2079, 18133], [2, 2, 0]]
[[2080, 18143], [2, 2, 0]]
[[2081, 18149], [2, 2, 0]]
[[2082, 18169], [2, 2, 0]]
[[2083, 18181], [2, 2, 0]]
[[2084, 18191], [2, 2, 0]]
[[2085, 18199], [2, 2, 0]]
[[2086, 18211], [2, 2, 0]]
[[2087, 18217], [2, 2, 0]]
[[2088, 18223], [2, 2, 0]]
[[2089, 18229], [2, 2, 0]]
[[2090, 18233], [2, 2, 0]]
[[2091, 18251], [4, 4, 0]]
[[2092, 18253], [6, 2, 0]]
[[2093, 18257], [2, 2, 0]]
[[2094, 18269], [3, 2, 0]]
[[2095, 18287], [2, 2, 0]]
[[2096, 18289], [2, 2, 0]]
[[2097, 18301], [2, 2, 0]]
[[2098, 18307], [2, 2, 0]]
[[2099, 18311], [2, 2, 0]]
[[2100, 18313], [2, 2, 0]]
[[2101, 18329], [2, 2, 0]]
[[2102, 18341], [2, 2, 0]]
[[2103, 18353], [2, 2, 0]]
[[2104, 18367], [2, 2, 0]]
[[2105, 18371], [2, 2, 0]]
[[2106, 18379], [2, 2, 0]]
[[2107, 18397], [2, 2, 0]]
[[2108, 18401], [2, 2, 0]]
[[2109, 18413], "S0"]
[[2110, 18427], [2, 2, 0]]
[[2111, 18433], [6, 2, 0]]
[[2112, 18439], [2, 2, 0]]
[[2113, 18443], [4, 4, 0]]
[[2114, 18451], [2, 2, 0]]
[[2115, 18457], [2, 2, 0]]
[[2116, 18461], [2, 2, 0]]
[[2117, 18481], "?"]
[[2118, 18493], [2, 2, 0]]
[[2119, 18503], [2, 2, 0]]
[[2120, 18517], [2, 2, 0]]
[[2121, 18521], [2, 2, 0]]
[[2122, 18523], [6, 2, 0]]
[[2123, 18539], [2, 2, 0]]
[[2124, 18541], [2, 2, 0]]
[[2125, 18553], [2, 2, 0]]
[[2126, 18583], [2, 2, 0]]
[[2127, 18587], [2, 2, 0]]
[[2128, 18593], [2, 2, 0]]
[[2129, 18617], [2, 2, 0]]
[[2130, 18637], [2, 2, 0]]
[[2131, 18661], [2, 2, 0]]
[[2132, 18671], [2, 2, 0]]
[[2133, 18679], [2, 2, 0]]
[[2134, 18691], [2, 2, 0]]
[[2135, 18701], [2, 2, 0]]
[[2136, 18713], "S0"]
[[2137, 18719], [2, 2, 0]]
[[2138, 18731], [4, 4, 0]]
[[2139, 18743], [2, 2, 0]]
[[2140, 18749], [2, 2, 0]]
[[2141, 18757], [2, 2, 0]]
[[2142, 18773], [2, 2, 0]]
[[2143, 18787], [2, 2, 0]]
[[2144, 18793], [2, 2, 0]]
[[2145, 18797], [2, 2, 0]]
[[2146, 18803], [2, 2, 0]]
[[2147, 18839], [2, 2, 0]]
[[2148, 18859], [2, 2, 0]]
[[2149, 18869], [2, 2, 0]]
[[2150, 18899], [2, 2, 0]]
[[2151, 18911], [2, 2, 0]]
[[2152, 18913], [2, 2, 0]]
[[2153, 18917], [2, 2, 0]]
[[2154, 18919], [2, 2, 0]]
[[2155, 18947], "?"]
[[2156, 18959], [2, 2, 0]]
[[2157, 18973], [2, 2, 0]]
[[2158, 18979], "S0"]
[[2159, 19001], [2, 2, 0]]
[[2160, 19009], [2, 2, 0]]
[[2161, 19013], [2, 2, 0]]
[[2162, 19031], [2, 2, 0]]
[[2163, 19037], [2, 2, 0]]
[[2164, 19051], [2, 2, 0]]
[[2165, 19069], [2, 2, 0]]
[[2166, 19073], [2, 2, 0]]
[[2167, 19079], [2, 2, 0]]
[[2168, 19081], [2, 2, 0]]
[[2169, 19087], [2, 2, 0]]
[[2170, 19121], [2, 2, 0]]
[[2171, 19139], "S0"]
[[2172, 19141], [2, 2, 0]]
[[2173, 19157], [3, 2, 0]]
[[2174, 19163], [2, 2, 0]]
[[2175, 19181], [2, 2, 0]]
[[2176, 19183], [2, 2, 0]]
[[2177, 19207], [2, 2, 0]]
[[2178, 19211], [4, 4, 0]]
[[2179, 19213], [2, 2, 0]]
[[2180, 19219], "S0"]
[[2181, 19231], [2, 2, 0]]
[[2182, 19237], [2, 2, 0]]
[[2183, 19249], [2, 2, 0]]
[[2184, 19259], [4, 4, 0]]
[[2185, 19267], [2, 2, 0]]
[[2186, 19273], [2, 2, 0]]
[[2187, 19289], [2, 2, 0]]
[[2188, 19301], [2, 2, 0]]
[[2189, 19309], [2, 2, 0]]
[[2190, 19319], [2, 2, 0]]
[[2191, 19333], [2, 2, 0]]
[[2192, 19373], [2, 2, 0]]
[[2193, 19379], "S1"]
[[2194, 19381], [2, 2, 0]]
[[2195, 19387], [2, 2, 0]]
[[2196, 19391], [2, 2, 0]]
[[2197, 19403], [2, 4, 0]]
[[2198, 19417], [2, 2, 0]]
[[2199, 19421], [2, 2, 0]]
[[2200, 19423], [2, 2, 0]]
[[2201, 19427], [2, 2, 0]]
[[2202, 19429], [2, 2, 0]]
[[2203, 19433], [2, 2, 0]]
[[2204, 19441], [4, 2, 0]]
[[2205, 19447], "S0"]
[[2206, 19457], [2, 2, 0]]
[[2207, 19463], [2, 2, 0]]
[[2208, 19469], [2, 2, 0]]
[[2209, 19471], [2, 2, 0]]
[[2210, 19477], [2, 2, 0]]
[[2211, 19483], [2, 2, 0]]
[[2212, 19489], [2, 2, 0]]
[[2213, 19501], [4, 2, 0]]
[[2214, 19507], "S0"]
[[2215, 19531], [2, 2, 0]]
[[2216, 19541], [2, 2, 0]]
[[2217, 19543], [2, 2, 0]]
[[2218, 19553], [2, 2, 0]]
[[2219, 19559], [2, 2, 0]]
[[2220, 19571], [2, 2, 0]]
[[2221, 19577], "S0"]
[[2222, 19583], [2, 2, 0]]
[[2223, 19597], [2, 2, 0]]
[[2224, 19603], "T0"]
[[2225, 19609], [2, 2, 0]]
[[2226, 19661], [2, 2, 0]]
[[2227, 19681], [2, 2, 0]]
[[2228, 19687], [2, 2, 0]]
[[2229, 19697], [2, 2, 0]]
[[2230, 19699], [2, 2, 0]]
[[2231, 19709], [2, 2, 0]]
[[2232, 19717], [2, 2, 0]]
[[2233, 19727], [2, 2, 0]]
[[2234, 19739], [2, 2, 0]]
[[2235, 19751], [2, 2, 0]]
[[2236, 19753], [2, 2, 0]]
[[2237, 19759], [2, 2, 0]]
[[2238, 19763], [2, 2, 0]]
[[2239, 19777], [2, 2, 0]]
[[2240, 19793], [2, 2, 0]]
[[2241, 19801], [2, 2, 0]]
[[2242, 19813], [2, 2, 0]]
[[2243, 19819], [2, 2, 0]]
[[2244, 19841], [2, 2, 0]]
[[2245, 19843], "S0"]
[[2246, 19853], [2, 2, 0]]
[[2247, 19861], [2, 2, 0]]
[[2248, 19867], [2, 2, 0]]
[[2249, 19889], [2, 2, 0]]
[[2250, 19891], [2, 2, 0]]
[[2251, 19913], [2, 2, 0]]
[[2252, 19919], [2, 2, 0]]
[[2253, 19927], [4, 2, 0]]
[[2254, 19937], [2, 2, 0]]
[[2255, 19949], [6, 2, 0]]
[[2256, 19961], [2, 2, 0]]
[[2257, 19963], [2, 2, 0]]
[[2258, 19973], "S0"]
[[2259, 19979], [2, 2, 0]]
[[2260, 19991], [2, 2, 0]]
[[2261, 19993], [2, 2, 0]]
[[2262, 19997], "S0"]
time = 11h, 38min, 29,910 ms.
\end{verbatim}
\end{multicols}
}

%\newpage
\section{Appendix}\label{seApp}
\begin{table}[h]
%{\small
{\small
\caption{\cite[Proposition 3.6 (i)]{EM73} for $p< 20000$}
%{\small
\begin{tabular}{l}
$p$: $\bQ(C_p)$ is not rational over $\bQ$\\\hline
47,79,167,191,223,239,263,359,367,383,431,439,463,479,503,\\
599,607,719,823,839,863,887,911,983,1031,1039,1087,1103,1223,1231,\\
1303,1319,1327,1367,1399,1439,1447,1487,1511,1543,1559,1583,1663,1759,1823,\\
1831,1847,1871,1879,2039,2063,2087,2111,2207,2239,2383,2399,2423,2447,2543,\\
2671,2687,2711,2767,2879,2903,2927,2999,3023,3119,3167,3191,3319,3343,3359,\\
3391,3407,3463,3559,3607,3623,3671,3767,3847,3863,3919,3967,4007,4079,4111,\\
4127,4271,4327,4391,4423,4463,4567,4583,4639,4679,4703,4759,4783,4799,4831,\\
4871,4919,4943,4967,5039,5087,5119,5231,5279,5303,5399,5431,5471,5479,5503,\\
5519,5591,5623,5639,5647,5711,5791,5807,5839,5879,5903,5927,6047,6079,6143,\\
6199,6263,6287,6311,6367,6599,6703,6719,6791,6863,6871,6911,6983,6991,7079,\\
7103,7127,7159,7207,7247,7487,7559,7583,7607,7639,7703,7727,7823,7879,7919,\\
7927,8039,8087,8111,8167,8231,8287,8311,8423,8431,8447,8543,8599,8647,8663,\\
8719,8783,8807,8831,8863,8887,8999,9007,9103,9239,9319,9391,9431,9463,9479,\\
9511,9623,9679,9719,9743,9767,9791,9839,9871,9887,9967,10007,10039,10079,10103,\\
10111,10159,10223,10247,10271,10303,10343,10391,10399,10463,10559,10607,10631,10663,10687,\\
10799,10847,10903,11047,11087,11119,11159,11239,11279,11311,11383,11399,11423,11447,11471,\\
11519,11527,11743,11783,11807,11839,11887,11903,11927,11959,12071,12119,12143,12239,12263,\\
12391,12479,12487,12503,12527,12647,12671,12703,12743,12791,12823,12911,12919,12959,12967,\\
12983,13007,13063,13103,13127,13327,13367,13399,13463,13487,13567,13679,13687,13711,13759,\\
13799,13831,13903,13967,13999,14071,14087,14143,14159,14207,14303,14327,14423,14431,14447,\\
14479,14503,14519,14543,14591,14639,14759,14767,14783,14831,14879,15199,15263,15271,15287,\\
15359,15383,15439,15527,15559,15647,15671,15727,15767,15791,15919,15959,15991,16007,16063,\\
16087,16103,16127,16223,16231,16319,16447,16487,16519,16567,16631,16703,16823,16879,16943,\\
17159,17167,17207,17231,17327,17359,17383,17471,17519,17599,17783,17791,17807,17863,17903,\\
17959,18047,18119,18143,18191,18223,18287,18311,18367,18439,18583,18671,18679,18743,18839,\\
18911,18959,19031,19079,19087,19183,19231,19319,19391,19447,19471,19543,19559,19583,19687,\\
19727,19759,19919,19991\\
\end{tabular}
}
\label{taEM1}
\end{table}

\begin{table}[h]
%{\small
{\small
\caption{\cite[Proposition 3.6 (ii)]{EM73} for $p< 20000$}
%{\small
\begin{tabular}{l}
$p$: $\bQ(C_p)$ is not rational over $\bQ$\\\hline
113,137,233,521,593,617,809,977,1033,1097,1129,1193,1289,1361,1489,\\
1553,1609,1777,1993,2129,2153,2281,2417,2441,2473,2609,2729,2833,2897,3049,\\
3089,3121,3209,3217,3433,3593,3761,3793,3881,4073,4241,4273,4297,4337,4457,\\
4561,4649,4657,4721,4817,4937,5009,5233,5297,5393,5417,5449,5521,5641,5737,\\
5897,6089,6217,6257,6353,6449,6473,6569,6577,6673,6737,6793,6833,6857,7121,\\
7177,7369,7433,7529,7537,7753,7793,7817,8009,8017,8081,8273,8297,8329,8369,\\
8521,8681,8689,8753,8849,8969,9041,9137,9161,9769,9833,9929,10289,10313,10321,\\
10889,10993,11057,11113,11177,11273,11497,11633,11657,11689,12041,12049,12073,12113,12329,\\
12433,12497,12553,12689,12713,12721,12809,13009,13297,13417,13513,13577,13649,13841,14033,\\
14057,14153,14249,14281,14321,14537,14633,14737,14929,15017,15217,15241,15313,15473,15497,\\
15569,15761,15817,15881,15889,16361,16369,16433,16529,16553,16649,16657,16937,17033,17041,\\
17257,17321,17393,17417,17449,17489,17609,17681,17737,18089,18097,18121,18257,18313,18353,\\
18481,19121,19249,19273,19433,19697,19753,19793,19889
\end{tabular}
}
\label{taEM2}
\end{table}

%%%%%%%%%%%%%%%%%%%%%%%%%%

\vspace*{5mm}
\noindent
Akinari Hoshi\\
Department of Mathematics\\
Niigata University\\
8050 Ikarashi 2-no-cho\\
Nishi-ku, Niigata, 950-2181\\
Japan\\
E-mail: \texttt{hoshi}@\texttt{math.sc.niigata-u.ac.jp}\\
Web: \texttt{http://mathweb.sc.niigata-u.ac.jp/\~{}hoshi/}
\end{document}